\documentclass[11pt,a4paper]{amsart}
\usepackage[margin=2.5cm]{geometry}                
\usepackage{centernot}
\usepackage{graphicx}
\usepackage{amssymb}
\usepackage{color}
\usepackage{bbm, dsfont}
\usepackage{hyperref}
\usepackage{upgreek}
\usepackage{relsize}

\hypersetup{
	colorlinks=true,
    linkcolor=red,
    citecolor=blue,
    filecolor=magenta,
	pdfborder={0 0 0},
	pdfborderstyle={/S/U/W 0},
}

\numberwithin{equation}{section}

\usepackage{verbatim}

\usepackage[dvipsnames]{xcolor}

\usepackage{ulem}
\newtheorem{theorem}{Theorem}[section]
\newtheorem{lemma}[theorem]{Lemma}

\newtheorem{proposition}[theorem]{Proposition}

\theoremstyle{definition}

\newtheoremstyle{myremark}
{}{}
{\itshape}
{}
{\itshape}
{.}
{ }
{}
\theoremstyle{myremark}

\theoremstyle{definition}
\newtheorem*{remark*}{Note}

\numberwithin{equation}{section}

\usepackage{mathrsfs}
\usepackage{verbatim}
\usepackage{amsmath,amssymb,amsthm}             
\usepackage{amsfonts}            
\usepackage{mathrsfs}            
\usepackage{algorithm}  
\usepackage{algorithmicx}  
\usepackage{algpseudocode}  
\usepackage{mathtools}
\usepackage{commath}
\usepackage{physics}
\usepackage{bm}
\usepackage{graphicx}

\usepackage{float}
\usepackage{listings}
\usepackage{subfigure}
\usepackage{multirow}
\usepackage{color}
\usepackage[export]{adjustbox} 
\usepackage{enumerate}
\usepackage{bbm}

\newcommand{\eps}{\varepsilon}
\newcommand{\mphi}{\widetilde\varphi}
\newcommand{\mB}{\widetilde B}
\newcommand{\Z}{\mathbb Z^2}
\newcommand{\mZ}{\widetilde{\mathbb Z}^2}
\newcommand{\R}{\mathbbm{R}}
\newcommand{\A}{\mathbbm{A}}

\newcommand{\cC}{\mathcal{C}}

\renewcommand{\P}{\mathbbm{P}}
\newcommand{\wt}{\widetilde}

\newcommand{\Lc}{\mathcal{L}}

\newcommand{\Dc}{\mathcal{D}}
\newcommand{\connect}[1]{\overset{#1}\longleftrightarrow}
\newcommand{\disconnect}[1]{\overset{#1}{\centernot\longleftrightarrow}}

\usepackage{fancyhdr,ifoddpage}
\fancypagestyle{runningheads}{  \fancyhf{}      \fancyhead[L]{    \checkoddpage    \ifoddpage
      Y.~Bi, Y.~Gao, X.~Li    \else
      One-arm probabilities for 2D GFF    \fi
  }  \fancyhead[R]{\thepage}}

\begin{document}	
	
	\title{One-arm probabilities for the two-dimensional metric-graph and discrete Gaussian free field}
	
	\author{Yijie Bi$^1$}
		\thanks{$^1$School of Mathematical Sciences, Peking University, \href{mailto:2200010925@stu.pku.edu.cn}{2200010925@stu.pku.edu.cn}}

	\author{Yifan Gao$^2$}
		\thanks{$^2$Institute for Theoretical Sciences, Westlake University, \href{mailto:gaoyifan75@westlake.edu.cn}{gaoyifan75@westlake.edu.cn}}

	\author{Xinyi Li$^3$}
		\thanks{$^3$Beijing International Center for Mathematical Research, Peking University, \href{mailto:xinyili@bicmr.pku.edu.cn}{xinyili@bicmr.pku.edu.cn}}

\begin{abstract}
    We study the one-arm probability in the level-set percolation of the discrete and metric-graph Gaussian free field (GFF) defined on a box with Dirichlet boundary conditions. For the metric-graph case, we establish asymptotic estimates on two one-arm probabilities of interest. For the discrete case, we show up-to-constants bounds on the point-to-bulk probability and demonstrate its difference from the metric-graph case.
\end{abstract}

\maketitle

\section{Introduction}
The Gaussian free field (GFF) is the canonical massless Gaussian model of a random function on a graph or domain and plays a central role in probability, statistical physics, and conformal field theory; its level sets (excursion sets) therefore provide a natural and rich model for studying geometric phase transitions and connectivity properties.

The study of level-set percolation of Gaussian free field has a long history; see for instance \cite{lebowitz1986percolation,bricmont1987percolation}. The long-range nature of the model makes this problem particularly interesting, but also difficult to analyze. In three and higher dimensions, much progress has been made in the last decade. To name but a few: finiteness \cite{rodriguez2013phase} and strict positivity \cite{drewitz2018sign} of the critical level, large deviation estimates of the disconnection probability \cite{sznitman2015disconnection}; (stretched)-exponential decay of off-critical one-arm probability \cite{popov2015decoupling,goswami2022radius}, sharpness of phase transition \cite{duminil2023equality}, etc. The ``continuum'' variant of this model, namely the metric-graph GFF, exhibits nice integrability feature (e.g.\
 the explicit two-point function of the sign cluster \cite{lupu2016loop} and the cluster capacity functional \cite{drewitz2022cluster}), which allows to identify the critical level, and understand the critical and near-critical behavior of this model (e.g.\ critical exponents \cite{drewitz2023critical}, arm probabilities \cite{drewitz2023arm, cai2025one, drewitz2025critical, cai2024one, cai2025heterochromatic}, cluster volume \cite{drewitz2024cluster}, chemical distance \cite{DW20}, crossing probability \cite{ding2022crossing}, quasi-multiplicativity \cite{cai2024quasi}, incipient infinite clusters \cite{cai2024incipient}, switching identity \cite{werner2025switching}).

The level sets of GFF in two dimensions have also been studied intensively in recent years.
The work \cite{ding2018chemical} considers chemical distances in terms of bulk annulus-crossing of the level set, while \cite{ding2022crossing} considers left-to-right crossing of the sign clusters for both discrete and metric-graph GFF in a rectangular box. In \cite{aru2020first}, the scaling limit of certain natural interfaces of 2D metric-graph GFF is identified as ${\rm SLE}_4$ level lines.
In \cite{drewitz2024percolation}, upper and lower bounds are established for the critical threshold with respect to the boundary annulus-crossing event under the discrete setting. There are also series of works aimed at understanding the two-sided level sets of 2D GFF \cite{gao2022chemical,aru2025percolation,gao2024percolation} and the arm probabilities \cite{gao2024percolationb,bi2025arm}. Despite all the existing literature, the level-set percolation of 2D GFF is still quite far from being well understood. 

In this work, we establish various results regarding one-arm probabilities for both 2D metric-graph and discrete GFF. For the metric graph, we obtain the asymptotics of the point-to-bulk probability in Proposition \ref{thm:metric_graph_bulk}. We also establish an upper bound on the point-to-boundary probability at non-negative levels that differs from previously-known asymptotics for negative levels, revealing a phase transition at level 0. For the discrete GFF, we show that the point-to-bulk probability is of the same order as the corresponding metric-graph quantity, but establish a gap in the pre-factor between the two cases.
We would like to emphasize that the asymptotics we established in both cases are also explicit in the level $h$.

\subsection*{Notation}

We begin with necessary notation and refer readers to Section \ref{sec:prelim} for precise definitions. For $x\in \mathbb{Z}^2$, we use $|x|_\infty$ to denote its $l_\infty$ norm.
For any $N\in\mathbb{N}$, let $B_N:=\{x\in\mathbb{Z}^2: |x|_\infty\leq N\}$ be the box centered at $0$ with side length $2N$, and let $\mB_N$ denote its metric-graph version obtained by joining each pair of neighboring vertices with a line segment of length $2$, following the convention of \cite{lupu2016loop}. We use $\partial_i B_N$ and $\partial B_N$ to denote the inner and outer boundaries of $B_N$, respectively. 

Let $\mphi_N$ be the GFF on the metric-graph $\wt B_{N+1}$ with Dirichlet boundary conditions (i.e., the values of $\mphi_N$ on $\partial_i B_{N+1}$ and line segments joining vertices in it are pinned to $0$). Let $\varphi_N$ denote the discrete GFF on $B_{N+1}$, defined by the restriction of $\mphi_N$ to $B_{N+1}$. For all $h\in\mathbb R$, we denote the level set of $\varphi_N$ by $E_N^{\ge h}:=\{x\in B_N:\varphi_{N,x}\ge h\}$. For all non-empty $A_1, A_2\subset\Z$, let $\{A_1\connect{\varphi_N\ge h}A_2\}$ be the event that there exists a nearest-neighbor path in $E_N^{\ge h}$ joining $A_1$ and $A_2$. We simply write $\{z\connect{\varphi_N\ge h}A_2\}$ if $A_1$ is the singleton $\{z\}$.

Throughout, the following notation will be heavily used. For two positive sequences $f_N$ and $g_N$, we write $f_N \lesssim g_N$ if there exists some universal constant $c>0$ such that $f_N\le cg_N$ for all $N\in\mathbb N$. Define $f_N\gtrsim g_N$ if $f_N^{-1} \lesssim g_N^{-1}$, and write $f_N\asymp g_N$ if $f_N\lesssim g_N$ and $f_N\gtrsim g_N$. For any $a,b>0$, define $$a\nabla b:=(a\vee b)\wedge1.$$

\subsection*{Results}
We start with a preliminary result for the metric-graph GFF $\mphi_N$ that provides up-to-constants estimates for the one-arm probabilities in the bulk at any level $h\in\R$.

\begin{proposition}\label{thm:metric_graph_bulk}
    The following estimates hold:
    \begin{equation}\label{eq:mdecay-}
            \mathbb P[0\stackrel{\mphi_N\ge h}{\longleftrightarrow} \partial B_{N/2}]\asymp
            \frac{|h|}{\sqrt{\log N}}\nabla\frac{1}{\sqrt{\log N}}, \  h\le 0;
    \end{equation}
    \begin{equation}\label{eq:mdecay+}
            e^{-a_1h^2}(\log N)^{-\frac12}\lesssim\mathbb P[0\stackrel{\mphi_N\ge h}{\longleftrightarrow} \partial B_{N/2}]\lesssim e^{-a_2h^2}(\log N)^{-\frac12},\ h>0.
    \end{equation}
    Above, $a_1,a_2>0$ are some universal constants.
\end{proposition}
We refer to Proposition~\ref{prop:0tok} for a stronger result, which considers connection to $\partial B_{rN}$ for any $r\in (0,\frac12)$ and provides similar estimates given the value of $\mphi_N$ at the origin.

Note that in \eqref{eq:mdecay-} and \eqref{eq:mdecay+}, the decay in $N$ is always $(\log N)^{-\frac12}$ for all $h\in\R$.
However, as we now show, the decay rate for the one-arm probability to the boundary  satisfies a phase transition across level $0$. In fact, when $h<0$, the ``boundary'' one-arm probability can be exactly determined by using \cite{lupu2018random}. More precisely, by \cite[Proposition~2]{lupu2018random},
\begin{equation}\label{eq:exact_formula}
    \mathbb P[0\stackrel{\mphi_N\ge h}{\longleftrightarrow} \partial B_{N}]=\P[|\mphi_{N,0}|\le |h|] \quad \text{ for all } h<0.
\end{equation}
Note that $\mphi_{N,0}$ is a centered Gaussian random variable with variance given by the Green's function $\wt G_N(0,0)\asymp \log N$ (see Section~\ref{subsec:GFF}). Hence, from \eqref{eq:exact_formula}, one has
\begin{equation}
    \mathbb P[0\stackrel{\mphi_N\ge h}{\longleftrightarrow} \partial_i B_N] \asymp \mathbb P[0\stackrel{\mphi_N\ge h}{\longleftrightarrow} \partial B_{N}] \asymp |h|(\log N)^{-\frac12} \quad \text{ for all } h<0.
\end{equation}
In order to get a non-trivial arm probability when $h>0$, we will consider the connection to $\partial_i B_N$ instead of $\partial B_N$.
In comparison, when $h>0$, the boundary one-arm probability decays at least polynomially fast in $N$.
\begin{proposition}\label{thm:metric-boundary}
    There exists a constant $c$ such that for all $h\ge 0$ and all $N$, 
    \begin{equation}\label{eq:mdecay1}
    \mathbb P[0\stackrel{\mphi_N\ge h}{\longleftrightarrow} \partial_i B_N]\lesssim (\log N)^{-1}N^{-ch^2}.
    \end{equation}
\end{proposition}

While this paper was being written, we were informed that Pierre-Fran\c{c}ois Rodriguez and Wen Zhang were considering similar problems in the massive metric-graph setting, 
where the boundary conditions are replaced by a killing with rate of order $N^{-2}$ on the metric-graph Brownian motion.
In their recent preprint \cite{rodriguez2025anomalousscalinglawtwodimensional}, they implemented a renormalization argument to prove an estimate on the one-arm probability for subcritical $h$ analogous to \eqref{eq:mdecay+}, and obtained the precise value of the pre-factor before $h^2$ (analogous to $a_1, a_2$ in \eqref{eq:mdecay+}) which takes the form of some version of the Brownian capacity of a line segment in the plane. 
\medskip

We now turn to the main result of this work, namely the ``bulk'' one-arm probability for the discrete GFF $\varphi_N$. We show that its asymptotics in $N$ remains of order $(\log N)^{-\frac12}$ as in the metric graph case. 
However, the discrete estimates requires much more efforts to establish due to issues from the discrete nature of the problem, such as that the exploration process associated with the level set will not stop exactly at the level $h$. This will be more clear in Section~\ref{sec:discrete}.

\begin{theorem}\label{thm:discrete}
    We have the following estimates:
    \begin{equation} \label{eq:dsub-}
        \mathbb P[0\stackrel{\varphi_N\ge h}{\longleftrightarrow} \partial B_{N/2}]\asymp
            \frac{|h|}{\sqrt{\log N}}\nabla\frac1{\sqrt{\log N}}, \ h\le 0;
    \end{equation}
    \begin{equation}\label{eq:dsub+}
            e^{-a_3h^2}(\log N)^{-\frac12}\lesssim\mathbb P[0\stackrel{\varphi_N\ge h}{\longleftrightarrow} \partial B_{N/2}]\lesssim e^{-a_4h^2}(\log N)^{-\frac12},\ h>0,
    \end{equation}
    where $a_3,a_4>0$ are some universal constants.
\end{theorem}

Moreover, we prove that the above one-arm probability differs from its counterpart in the metric graph by a constant multiple of $(\log N)^{-\frac12}$. This result bears some resemblance to \cite[Theorem 1.2]{ding2022crossing}, where it is proved that the left-right crossing probabilities of the sign cluster of the discrete  and metric-graph GFF's differ by a positive constant.

\begin{theorem}\label{thm:comparison}   
    For any $h\in\mathbb{R}$, there exists some $c=c(h)>0$ such that
\begin{equation}\label{eq:arm-diff}
\mathbb P[0\stackrel{\varphi_N\geq h}{\longleftrightarrow}\partial B_{N/2}]-\mathbb P[0\stackrel{\mphi_N\geq h}{\longleftrightarrow}\partial B_{N/2}]\ge c(\log N)^{-\frac12}.
\end{equation}
\end{theorem}

As an intermediate step in the proof of \eqref{eq:arm-diff}, we also establish the following RSW-type bounds, which is of independent interest: with positive probability uniformly in $N$, there exists a circuit (i.e., loop) above any fixed level in any annulus $A_{\alpha N,\beta N}:=B_{\beta N}\setminus B_{\alpha N}$ for $0<\alpha<\beta<1$. 

\begin{proposition}\label{prop:circuit}
    For all $h\in\R$ and  $0<\alpha<\beta<1$, there exists $c=c(h,\alpha,\beta)>0$ such that for all large $N$,
    \begin{equation}\label{eq:dC1}
            \mathbb P[ \text{there exists a circuit in $A_{\alpha N,\beta N}$ that surrounds $B_{\alpha N}$, on which $\varphi_N\ge h$} ]>c,
    \end{equation} 
and hence \begin{equation}\label{eq:annulus-crossing}
        \mathbb P[ B_{\alpha N}\stackrel{\varphi_N\le h}{\longleftrightarrow} \partial B_{\beta N}]<1-c.
    \end{equation}
\end{proposition}
Note that \eqref{eq:annulus-crossing} confirms the speculation in \cite[Remark~1]{ding2018chemical}.

Incidentally, our results indicate the following bound on the chemical distance on the discrete level set across a macroscopic annulus given connectivity, parallel to the result on the metric graph \cite[Theorem 1]{DW20}). 

For $h\in\mathbb{R}$, let $\Dc_{N,h}$ denote the graph distance on the level set $E_N^{\ge h}$. For any $h\in\mathbb R$, $0<\alpha<\beta<\gamma<1$, and $\eps>0$, there exists a constant $c$ such that 
\begin{equation}\label{eq:chemical}
         \limsup_{N\to\infty} \P[\Dc_{N,h}(B_{\alpha N}, \partial B_{\beta N})>c N (\log N)^{1/4}\mid B_{\alpha N}\stackrel{\varphi_N\ge h}{\longleftrightarrow} \partial B_{\gamma N}]\le \eps.
\end{equation}     
Note that such a result has already been speculated in \cite[Section 1.3]{DW20}.

\subsection*{Proof strategy}

We now briefly outline the main ideas of the proofs. The central tool throughout the paper is the exploration martingale, introduced in \cite{DW20}. In both the metric-graph and discrete settings, we construct a continuous-time exploration process of the level-set cluster starting from the origin. Along this exploration, we track the evolution of the GFF on a finite set of vertices and consider the associated martingale given by the conditional expectation of a suitable observable, conditioned on the current explored boundary.

This approach allows us to reformulate the one-arm event in terms of a constraint on the trajectory of the exploration martingale: namely, the event that the martingale remains above a certain threshold over some time interval. This probabilistic reformulation is more amenable to analysis thanks to the Dubins-Schwarz Theorem (see Lemma~\ref{lem:dubins}), and enables us to derive sharp estimates.

The metric-graph case is comparatively simpler. In this setting, the exploration process enjoys an exact stopping property: the exploration terminates precisely when all boundary points of the cluster simultaneously reach the level $h$ (see \eqref{eq:terminal_condition}). This structural feature allows for a relatively direct analysis of the associated martingale.

In contrast, the discrete setting lacks such an exact stopping mechanism. To overcome this difficulty, we adapt an entropic repulsion argument from \cite{ding2022crossing} (see Proposition \ref{prop:cH_bound}). Roughly speaking, we show that with high probability, the harmonic average of the GFF along the boundary of the explored cluster does not fall significantly below the level $h$. This substitute control compensates for the absence of the exact stopping property and allows us to derive an upper bound on the discrete one-arm probability via a similar martingale argument. The corresponding lower bound follows more directly, as the discrete one-arm event is implied by its metric-graph counterpart.

\subsection*{Organization}
The rest of the paper is organized as follows. In Section 2, we introduce the notation used in this paper, together with some key tools (in particular the exploration martingale introduced in \cite{DW20}, from which we have drawn a lot of inspiration, and the connection between (metric-graph) GFF level sets, the loop soup and the excursion cloud, established in \cite{lupu2016loop} and \cite{aru2020first}). We will also establish some preliminary results in Section 2. In particular, we prove a stronger version of Proposition \ref{prop:circuit} (c.f.\ Proposition \ref{prop:circuit'}). In Section 3, we prove the results on the metric-graph GFF, namely Propositions \ref{thm:metric_graph_bulk} and \ref{thm:metric-boundary}. Finally, in Section 4, we focus on the discrete case, establishing Theorems \ref{thm:discrete} and \ref{thm:comparison} and sketching a proof of \eqref{eq:chemical}.

\medskip

\section{Preliminaries}\label{sec:prelim}

For $v=(v_1,v_2)\in\mathbb R^2$, let $|v|_\infty:=|v_1|\vee|v_2|$ and $|v|:=\sqrt{v_1^2+v_2^2}$ be its $l^\infty$ and $l^2$ norms respectively.
For any nonempty $A_1,A_2\subset\mathbb R^2$, define the $l^\infty$-distance between them by ${\rm dist}(A_1,A_2):=\inf\{|x_1-x_2|_\infty:x_1\in A_1,x_2\in A_2\}$, and write ${\rm dist}(x,A_2)$ if $A_1=\{x\}$ is a singleton. For all $x,y\in\mathbb Z^2$ and $r>0$, we say that $x,\ y$ are nearest neighbors and write $x\sim y$ if $|x-y|=1$. We define the discrete $r$-neighborhood of $x$ by $B(x,r)\coloneq\{y\in\mathbb Z^2:|x-y|_\infty\leq r\}$, and write $B_r=B(0,r)$. For any $0<r<r'<\infty$, define $A_{r,r'}=B_{r'}\setminus B_r$. For $A\subset\mathbb Z^2$, define its inner boundary by $\partial_i A=\{x\in A:\exists\ y\in A^c,x\sim y\}$, and its outer boundary by $\partial A=\{x\in A^c:\exists\ y\in A,x\sim y\}$. 
For any non-empty $A\subset \Z$, let $\wt A$ denote its metric graph version which is obtained by joining each pair of neighboring vertices by a line segment of length $2$. The unlabeled constants $c,c',c''$, etc.\ may vary from line to line, whereas the labeled constants $c_i,c_i'$ etc.\ are fixed throughout each {\it section}. Unless stated otherwise, all constants are universal and do not rely on any additional parameters.

\subsection{Basic results on Brownian motions and martingales}
We start with an useful estimate  for standard one-dimensional Brownian motions.
Let 
\begin{equation}\label{eq:surv}
\overline\Phi(s):=\int_s^\infty\frac1{\sqrt{2\pi}}e^{-\frac12u^2}du    
\end{equation}
 denote the tail probability of the standard normal distribution. 
  We will use the following estimate on the hitting probability of a line with any given slope; see e.g.\  \cite[equation (2.0.2) in Part II]{borodin2012handbook} or \cite[Proposition 15]{DW20}. 
 
\begin{lemma}\label{lem:compute}
Suppose $(B_t)_{t\ge 0}$ is a standard one-dimensional Brownian Motion.
    For $m\in\mathbb R$ and $b>0$, let $\tau=\inf\{t>0:B_t\le mt-b\}$. For any $T>0$,
    $$\mathbb P[\tau\le T]=1-\psi(m,b,T),$$
    where
    \begin{equation}\label{eq:psidef}
        \psi(m,b,T):=\overline\Phi\left(m\sqrt T-\frac{b}{\sqrt T}\right)-e^{2bm}\overline\Phi\left(m\sqrt T+\frac{b}{\sqrt T}\right).
    \end{equation}
\end{lemma}
    For some constants $c,c'>0$, the following asymptotics hold:
    \begin{equation}\label{eq:asymp+}
        \frac {be^{-cm^2T}}{\sqrt T}\lesssim\psi(m,b,T)\lesssim\frac{be^{-c'm^2T}}{\sqrt T},\;\mbox{when }m>0, 0<b\le\sqrt T;
    \end{equation}
    \begin{equation}\label{eq:asymp-}
    \psi(m,b,T)\asymp |mb|\nabla\frac b{\sqrt T},\mbox{ when }m\le0.
    \end{equation}
\hspace{\fill}

Next, we turn to the Brownian motion on the metric graph.
The Brownian motion $X$ can be defined on $\mZ$ in a certain way such that it behaves like a standard Brownian motion in the interior of edges and ``chooses to make excursions on each incoming edge uniformly at random'' on vertices (see \cite[Section 2]{lupu2016loop} for a rigorous definition). For $x\in\mZ$, we denote the law of $X$ with $X_0=x$ by $P^x$ and the corresponding expectation by $E^x$. For all non-empty $A\subset\mZ$, the first hitting time of $A$ by $X$ is given by
$$\tau_A:=\inf\{t\ge 0:X_t\in A\}.$$ 

In \cite[Section 2]{lupu2016loop}, Lupu also introduced 
the Green's function for the Brownian motion on $\mZ$ killed upon exiting some subset $A$ of $\mZ$. We denote it by
$\wt G_A(x,y)$, for all $x,y\in A$, and abbreviate $\wt G_N\coloneq \wt G_{\wt B_{N+1}}$. For $V\subset\Z$ non-empty, 
let $G_V$ denote the Green's function for the simple random walk on $\Z$ killed upon exiting $V$. Then, it is argued in \cite{lupu2016loop} that $G_V$ is the restriction of
${\wt G}_{\widetilde{V\cup\partial V}}$ to $V$.

For $A\subset \wt B_{N+1}$, and $x\in\wt B_{N+1}, y\in A$, define the harmonic measure of $A$ from $x$ relative to $\mB_{N+1}$ at $y$ by
\begin{equation}\label{eq:intro_HN}
    H_N(x,y;A)=P^x[X_{\tau_{A\cup(\wt B_{N+1})^c}}=y],
\end{equation}
and write $H_N(U,y;A)=\sum_{x\in U}H_N(x,y;A)$ for all $U\subset B_N$. For any subset $D\subset A$, write 
$$H_N(U,D;A)=\sum_{y\in D}H(U,y;A),$$
where the equality is well defined because  the number of $y$'s such that $H(U,y;A)>0$ is obviously finite. For simplicity, we write $H_N(U,A)=H_N(U,A;A)$, and write $H_N(x,A)=H_N(\{x\},A)$ for the harmonic measure of $A$ relative to $\mB_{N+1}$ seen from $x$. Finally, we set the average harmonic measure from $U$ as follows:
\begin{equation}\label{eq:ave-harm}
    \overline H_N(U,\cdot)=\frac1{|U|}H_N(U,\cdot).
\end{equation}

We will frequently use the following version of Beurling estimate  stated for Brownian motion on the metric graph.
\begin{lemma}\label{lem:beurling}
    There exists $c>0$ such that for any connected $A\subset\Z$ with $0\in A$ and $A\cap\partial B_n\neq\emptyset$, we have 
    $P^x[\tau_{\partial B_n}<\tau_A]\le cn^{-\frac12}$. 
\end{lemma}
In fact, Lemma~\ref{lem:beurling} follows directly from the corresponding result for simple random walk (e.g.\ \cite[Theorem 6.8.3]{lawler2010random}), 
since $X$ restricted to $\Z$ can be viewed as a simple random walk.

Moreover, by \cite[Proposition 6.4.1]{lawler2010random}, for all $0<k<N$ and $x\in A_{k,N}$, we have
$$H_N(x,B_k)=\frac{\log N-\log |x|_\infty+O(k^{-1})}{\log N-\log k}.$$
Combining this estimate and Lemma \ref{lem:beurling}, one immediately obtains
\begin{lemma}\label{lem:hitting}
        There exist $c,c'>0$ such that for any $x\in\wt A_{\frac58N,\frac78N}$ and $S$
    with $0\in S\subset\wt B_{N/2}$, we have 
    $$c(\log N-\log{\rm diam}(S))^{-1}\le H_N(x,S)\le c'(\log N-\log{\rm diam}(S))^{-1},$$
    where ${\rm diam} (S):=\sup_{x,y\in S}|x-y|_\infty$. 
\end{lemma}

\hspace{\fill}

Finally, we cite a version of the Dubins-Schwarz Theorem \cite[Theorem~1.7 in Chapter~5]{revuz2013continuous}, which will be useful in estimating deviation probabilities of exploration martingales (see Section~\ref{subsec:exploration}). For any continuous martingale $M$, let $\langle M\rangle_.$ denote its quadratic variation process.
\begin{lemma}\label{lem:dubins}
    Let $M$ be a continuous martingale, $T_t=\inf\{s:\langle M\rangle_s>t\}$, and $W$ be the following process:
    $$W_t=\begin{cases} M_{T_t}-M_0,&t<\langle M\rangle_\infty;\\M_\infty-M_0,&t\ge\langle M\rangle_\infty.\end{cases}$$
    Then $W$ is a standard Brownian motion stopped at $\langle M\rangle_\infty$.
\end{lemma}

\subsection{The Gaussian free field (GFF)}\label{subsec:GFF}

We refer to \cite{lupu2016loop} for a more comprehensive introduction. Define $\mphi_N$ as the metric-graph GFF on $\widetilde B_{N+1}$ with Dirichlet boundary conditions, i.e., the centered continuous Gaussian field on $\wt B_{N+1}$ with covariance function $\wt G_N$ and boundary values $0$ on $\partial_i B_{N+1}$. Define the corresponding discrete GFF $\varphi_N$ by the restriction of $\mphi_N$ to 
$B_{N+1}$. 
For $h\in\mathbb R$, define the level set of $\mphi_N$ above level $h$ by $$\widetilde E_N^{\ge h}=\{x\in \widetilde B_{N+1}:\mphi_{N,x}\ge h\},$$
and set $E_N^{\ge h}=\widetilde E_N^{\ge h}\cap\mathbb Z^2$. 
We define the level sets below $h$ by $\widetilde E_N^{< h}$ and define $E_N^{< h}$ similarly. 
For $A_1,A_2\subset\mathbb Z^2$, let $\{A_1\stackrel{\varphi_N\geq h}{\longleftrightarrow}A_2\}$ be the event that $A_1$ and $A_2$ are connected by a nearest-neighbor path on $\mathbb Z^2$ that is contained in $E^{\ge h}_N$. For $A_1,A_2\subset\mZ$, let $\{A_1\connect{\mphi_N\ge h}A_2\}$ be the event that there exists a connected component of $\widetilde E_N^{\ge h}$ intersecting both $A_1$ and $A_2$. If $A_i=\{x_i\}$ is a singleton, we replace $\{x_i\}$ with $x_i$ in the notation for brevity.

For $u,v\in\widetilde E_N^{\ge h}$ (resp.\ $E_N^{\ge h}$), we let the chemical distance $\widetilde  D_{N,h}(u,v)$ (resp.\ $D_{N,h}(u,v)$) be the graph distance between $u$ and $v$ in $\widetilde E_N^{\ge h}$ (resp.\ $E_N^{\ge h}$), with $\widetilde D_{N,h}(u,v)=\infty$ (resp.\ $D_{N,h}(u,v)=\infty$) if $u,v$ are not connected in $\widetilde E_N^{\ge h}$ (resp.\ $E_N^{\ge h}$). For any two subsets $A,B\subset\widetilde B_{N+1}$, we let $\widetilde D_{N,h}(A,B)=\inf\{\widetilde D_{N,h}(u,v):u\in A,v\in B\}$; for $A,B\subset B_{N+1}$, $D_{N,h}(A,B)$ is similarly defined.

Next, we state and prove a useful result on the moments of the GFF under certain conditions.
For $I\subset B_N$, we define the harmonic support of $I$ by
\begin{equation}\label{eq:intro_D}
    D(I):=\{x\in I:H_N(\partial_i B_N,x;I)>0\},
\end{equation}
and define $\mathcal F_I=\sigma\{\varphi_{N,x},\ x\in I\}$. For any random variable $X$, sigma algebra $\mathcal F$ and event $A$ in the same probability space, and any $k\in\mathbb N$, define the conditional $k$-th (central) moment of $X$ given $A$ and $\mathcal F$ by
\begin{equation}\label{eq:intro_muk}
    \mu^{(k)}[X\mid A, \mathcal F]:=\mathbb E\left[\left|X-\mathbb E[X\mid A, \mathcal F]\right|^{k}\,\Big|\,A,\mathcal F \right].
\end{equation}
The following lemma controls moments of the harmonic average of the GFF on $D(I)$, which follows from the results in \cite{ding2022crossing}.
\begin{lemma}\label{lem:moment}
    For $I\subset B_{N/2}$, let $I^+$ and $I^-$ be a partition of $I$ satisfying that for all $x\in I^-$, there is some $x'\in I^+$ such that $x\sim x'$. 
        For $U\subset A_{\frac58N,\frac78N}$,
        define $$Y:=\sum_{x\in D(I)}\overline H(U,x;I)\varphi_{N,x}.$$
    For any $h\in \R$, define the event $A=\{\varphi_{N,x}< h: x\in I^-\}\cap \{\varphi_{N,x}\ge h: x\in I^+\}$. Then, for any positive integer $k$, 
        \begin{equation}
        \mu^{(2k)}[Y\mid A,\ \mathcal F_{I^+}]\le(2k-1)!!\cdot16^k\,\xi(U,I)^k\, \overline H_N(U,I)^{2k},
    \end{equation}
    where 
    \begin{equation}\label{eq:xiI}
        \xi(U,I):=\sup_{x\in I\backslash D(I)} H_N(U,x;I\backslash D(I))/ H_N(U,I).
    \end{equation}
\end{lemma}
\begin{proof}
    By \cite[Corollary 4.4]{ding2022crossing}, we obtain
    $${\rm Var}[Y\mid\mathcal F_{I^+}]\le 16\,\xi(U,I)\cdot \overline H_N(U,I)^2.$$
    Since the conditional law of $Y$ given $\mathcal F_{I^+}$ is still Gaussian (Markov property of GFF), we have $$\mu^{(2k)}[Y\mid\mathcal F_{I^+}]\le (2k-1)!!\cdot16^k\xi(U,I)^k \overline H_N(U,I)^{2k}.$$
    The lemma then follows from the fact that
    $$\mu^{(2k)}[Y\mid A,\ \mathcal F_{I^+}]\le \mu^{(2k)}[Y\mid\mathcal F_{I^+}],$$
    which can be deduced from the Brascamp-Lieb inequality \cite[Theorem 5.1]{brascamp1976extensions}, similar to the proof of \cite[Lemma 4.3]{ding2022crossing} (we refer the reader to it for details).
\end{proof}

\subsection{Exploration martingales}\label{subsec:exploration}
We now briefly recall the exploration martingales associated with level sets of GFF, which serves a key tool in \cite{ding2018chemical,DW20,ding2022crossing}. 
For $S\subset B_N$, we define the exploration of the metric-graph level set $\widetilde E_N^{\ge h}$, $(\wt{\mathcal I}_t)_{t\ge 0}$, with source $\wt{\mathcal I}_0=S$ by
$$\wt{\mathcal I}_t=\{u\in\mB_{N+1}:\widetilde D_{N,h}(u,S)\le t\}.$$ 
In the discrete case, we will use $(\mathcal I_t)_{t\ge0}$ to denote the exploration process associated with the discrete level set $E_N^{\ge h}$ and source $S$, which is defined as follows.
We begin by setting $\mathcal V_0=S$, $\mathcal A_0=\mathcal V_0\cap E_N^{\ge h}$ and $\mathcal B_0=\mathcal V_0\cap E_N^{< h}$. For integer $k\ge1$, we let
\begin{align*}
    \mathcal A_k=\;&\{v\in (B_{N+1}\setminus\mathcal V_{k-1})\cap E_N^{\ge h}:\exists u\in\mathcal A_{k-1},u\sim v\},\\
    \mathcal B_k=\;&\{v\in (B_{N+1}\setminus\mathcal V_{k-1})\cap E_N^{< h}:\exists u\in\mathcal A_{k-1},u\sim v\},\\
    \mathcal V_k=\;&\mathcal V_{k-1}\cup\mathcal A_{k-1}\cup\mathcal B_{k-1},
\end{align*}
and $\mathcal I_k$ be the induced metric graph on $\mathcal V_k$. The explored set $\mathcal I_t$ for any $t>0$ is defined through linear interpolation. Hence, $\mathcal I_t$ is indeed a metric graph for any $t$. 

For $U\subset\mB_N$ finite, we define the observables $X_U$ and $\overline X_U$ by
\begin{equation}\label{eq:observable}
        X_U\coloneq\sum_{x\in U}\mphi_{N,x} \  \text{ and } \ 
        \overline X_U\coloneq\frac{X_U}{|U|}.
\end{equation}
Then, we define the exploration martingales corresponding to observables $X_U$ by
\begin{equation}\label{eq:Mt}
    M_{U,t}\coloneq\mathbb E[X_U\mid\mathcal F_{\wt{\mathcal I}_t}]
\end{equation}
in the metric-graph case, and
$$M_{U,t}\coloneq\mathbb E[X_U\mid\mathcal F_{\mathcal I_t}]$$
in the discrete case. Here, with slight abuse of notation we use the same symbol for both cases, but it would be clear from the context. We similarly define the renormalized exploration martingale $\overline M_U$ by replacing $X_U$ with $\overline X_U$ in the definitions above.

\subsection{Isomorphism theorem and circuit estimates}
In this section, we first recall a powerful isomorphism theorem for GFF from \cite{aru2020first}, and then use it to prove Proposition~\ref{prop:circuit}.
We refer readers to \cite[Section 2]{aru2020first} for rigorous introduction of Brownian loop soups and Brownian excursions on metric graphs. Let $\wt\Lc^{1/2}_N$ be the loop soup in $\wt B_{N+1}$ with intensity $\frac12$.
For $h>0$, let $\Xi^h_N$ be the excursion cloud in $\wt B_{N+1}$ with intensity $\frac{h^2}{2}$, which is a random set of Brownian excursions in $\wt B_{N+1}$ with endpoints in $\partial \wt B_{N+1}$. 
We consider all the clusters made by loops and excursions inside independent union of $\wt\Lc^{1/2}_N$ and $\Xi^h_N$. Let $\wt\A^{-h}_N$ be the union of $\partial\wt B_{N+1}$ and topological closures of the clusters that contain at least one excursion (i.e.\ are connected to $\partial\wt B_{N+1}$). According to \cite[Propositions~2.4 and~2.5]{aru2020first}, we have the following description of level sets of metric-graph GFF.

\begin{lemma}\label{lem:description}
    Let $h>0$. There is a coupling such that the clusters of loops and excursions (in the union of $\wt\Lc^{1/2}_N$ and $\Xi^h_N$) correspond to the sign clusters of $\mphi_{N}+h$ in the following way. 
    \begin{itemize}
        \item The zero set of $\mphi_{N}+h$ is exactly the set of vertices not visited by any loop or excursion. 
        \item $\wt\A^{-h}_N$ is the set of vertices that can be connected to $\partial\wt B_{N+1}$ via a continuous path in $\wt E^{\ge-h}_N$.
        \item For those clusters that contain no excursion, each has equal probability to be a cluster in $\wt E_N^{>-h}$ or $\wt E_N^{<-h}$, which does not intersect $\partial\wt B_{N+1}$.
    \end{itemize}
\end{lemma}

With Lemma~\ref{lem:description}, we are now ready to show a version of Proposition~\ref{prop:circuit} on the metric graph. As we will see immediately, it implies the discrete one directly. 
For $1\le k<l\le N$ and $h\in\mathbb R$, let $\wt\cC_{N,k,l}^h$ denote the event that there exists a continuous circuit inside the annulus $\wt A_{k,l}$ that surrounds $\wt B_k$, on which $\mphi_N\ge h$.

\begin{proposition}\label{prop:circuit'}
    For all $h\in\R$ and  $\eps\in(0,1)$, there exists $c=c(\eps,h)>0$ such that for all $(1+\eps)k<l\le (1-\eps)N$ with $k$ sufficiently large, we have $\P[\wt\cC_{N,k,l}^h]\ge c$.
\end{proposition}
\begin{proof}
    We rely on the loop-soup description of level sets of metric-graph GFF, as explained in Lemma~\ref{lem:description}.  Note that we only need to prove the result for $h>0$ by monotonicity, which we assume tacitly below. Recall the notation introduced above Lemma~\ref{lem:description}. Let $H_{N,k,l}$ be the event that there exists a loop of $\wt\Lc^{1/2}_N$ inside $\wt A_{k,l}$ that surrounds $\wt B_k$. 
    By \cite[Lemma~4.6]{bi2025arm}, there exists $c_1=c_1(\epsilon)>0$,
    \begin{equation}\label{eq:Hlb}
        \P[H_{N,k,l}]\geq c_1.
    \end{equation}
    Note that with the help of the coupling in Lemma~\ref{lem:description}, it follows that
    \begin{align}\label{eq:CH}
        \mathbb P[ \wt\cC^h_{N,k,l} ]\ge \frac 12 \mathbb P[ \wt\A^{-h}_N\cap \wt B_{(1-\epsilon)N}=\emptyset ,H_{N,k,l}],
    \end{align}
     for the following reason: On the event $\{\wt\A^{-h}_N\cap \wt B_{(1-\epsilon)N}=\emptyset\}\cap H_{N,k,l}$, let $\gamma$ be any Brownian loop in $\wt A_{k,l}$ that surrounds $\wt B_k$, then the cluster of $\gamma$ contains no excursion and has probability $\frac12$ to be a cluster in $\wt E_N^{\ge h}$ under the coupling (this explains the prefactor $\frac12$ appearing in \eqref{eq:CH}), and therefore the loop $\gamma$ provides a continuous circuit inside $\wt A_{k,l}\cap \wt E_N^{\ge h}$ that surrounds $\wt B_k$. Altogether, this implies \eqref{eq:CH} by using the symmetry of $\mphi_N$. 

            Moreover, the event $\{\wt\A^{-h}_N\cap \wt B_{(1-\eps)N}=\emptyset\}$ is independent of loops inside $\wt B_l$, and hence of $H_{N,k,l}$. Hence, by \eqref{eq:Hlb} and \eqref{eq:CH} it follows that 
    \begin{equation}\label{eq:CH'}
        \mathbb P[ \wt\cC^h_{N,k,l}] 
        \ge \frac{c_1}{2}\,\mathbb P[ \wt\A^{-h}_N\cap \wt B_{(1-\epsilon)N}=\emptyset ].
    \end{equation}

    It remains to give a lower bound for $\mathbb P[ \wt\A^{-h}_N\cap \wt B_{(1-\epsilon)N}=\emptyset ]$. For this, by \cite[Lemma~4.13]{aru2020first}, there exist $c_2=c_2(\epsilon)>0$ and $\delta=\delta(\epsilon)<\eps$ such that with probability at least $c_2$, there does not exist a cluster in $\wt\Lc^{1/2}_N$ that intersects both $\wt B_{(1-\epsilon)N}$ and $\partial\wt B_{(1-\delta)N}$. Moreover, the total mass of excursions in $\wt B_{N+1}$ that intersect $\partial\wt B_{(1-\delta)N}$ is bounded from above uniformly in $N$ (but depending on $\epsilon$ and $h$). Hence, with probability at least $c_3(\epsilon,h)$, no excursion reaches $\partial\wt B_{(1-\delta)N}$. With everything combined, we conclude that $\mathbb P[ \wt\A^{-h}_N\cap \wt B_{(1-\epsilon)N}=\emptyset ]\ge c_2c_3$ and hence $\mathbb P[ \wt\cC^h_{N,k,l} ]\ge \frac{c_1c_2c_3}{2}$, as desired.
\end{proof}

\begin{proof}[Proof of Proposition \ref{prop:circuit}]
    Noting that a continuous circuit in $\wt E_N^{\ge h}$ always contains a discrete circuit in $E_N^{\ge h}$, Proposition \ref{prop:circuit} follows by taking $k=\alpha N,\ l=\beta N$ and $\eps<(1-\beta)\wedge\frac{\beta-\alpha}\alpha$ in Proposition \ref{prop:circuit'}.
\end{proof}

\section{The Metric-Graph Case}
This section is dedicated to the proof of Propositions~\ref{thm:metric_graph_bulk} and~\ref{thm:metric-boundary} concerning bulk and boundary one-arm probabilities on the metric graph, respectively. We start with the proof of Proposition~\ref{thm:metric_graph_bulk}, which will be a direct consequence of the following stronger estimate for the one-arm probability given the value of GFF at the origin. 
To state the result, for any $h\in\R$, $x>0$, and $r\in(0,1/2]$, we define
\begin{equation}\label{eq:mkdecay}
    g_{N,r}^h(x):=\mathbb P\Big[0\stackrel{\mphi_N\ge h}{\longleftrightarrow} \partial B_{rN} \mid \mphi_{N,0}=h+x\sqrt{\log N}\Big],
\end{equation}
where we recall that $\mphi_N$ is the GFF on the metric graph $\widetilde B_{N+1}$ with Dirichlet boundary conditions.

\begin{proposition}\label{prop:0tok}
    For any $h\le 0$, $x>0$, and $r\in(0,1/2]$, we have
    \begin{equation}\label{eq:gn-}
        g_{N,r}^h(x)\asymp\frac{x|h|}{\sqrt{\log N}}\nabla\frac{x|\log r|^{\frac12}}{\sqrt{\log N}}.
    \end{equation}
    Moreover, there exist $c,c',c''>0$ such that if $h>0, x\le c''\sqrt{\log N}\ |\log r|^{-\frac12}$, setting $K=K(h,r):=h|\log r|^{-\frac12}$, we have
    \begin{equation}\label{eq:gn+}
        xe^{-cK^2}\sqrt{\frac{|\log r|}{\log N}}\lesssim g_{N,r}^h(x)\lesssim xe^{-c'K^2}\sqrt{\frac{|\log r|}{\log N}}.
    \end{equation}
\end{proposition}

\begin{proof}
We follow a similar strategy as in the proof of \cite[Lemma 24]{DW20} by analyzing the corresponding exploration martingale. However, we will keep track of the dependence on $h,x,r$ more carefully. 
Precisely, we consider the exploration process $(\widetilde{\mathcal I}_t)_{t\ge0}$ with source $\widetilde{\mathcal I}_0=\{0\}$ on $\widetilde E_N^{\ge h}$. 
   Let $U=\partial B_{\lfloor \frac{3}{4}N\rfloor}$, and consider the exploration martingale $\overline M=\overline M_U$ introduced in \eqref{eq:Mt}. Let $\sigma=\inf\{t\ge0:\widetilde{\mathcal I}_t\cap B_{rN}\neq\emptyset\}$, then $\{0\stackrel{\mphi_N\ge h}\longleftrightarrow\partial B_{rN}\}=\{\sigma<\infty\}$. Recall the definition of $\overline H_N$ from \eqref{eq:ave-harm} and observe that almost surely,
    \begin{equation}\label{eq:terminal_condition}
    \begin{aligned}
        \overline M_t> h\overline H_N(U,\widetilde{\mathcal I}_t)\quad\mbox{ for all }\quad 0\le t<\sigma;\\
        \overline M_\infty=h\overline H_N(U,\wt{\mathcal I}_\infty)\quad\mbox{ when }\quad\sigma=\infty.
    \end{aligned}
    \end{equation}
    By Lemma \ref{lem:hitting}, there exist universal constants $c_1,c_1'>0$ such that on the event $\{\sigma<\infty\}$, we have
    \begin{equation}\label{eq:c1c1'}
        c_1|\log r|^{-1}\le \overline H_N(U,\widetilde{\mathcal I}_\sigma)\le c_1'|\log r|^{-1}.
    \end{equation}
    Define $\sigma_1:=\inf\{t\ge0:\overline H_N(U,\widetilde{\mathcal I}_t)\ge c_1|\log r|\}$, and similarly define $\sigma_1'$ with $c_1$ replaced by $c_1'$. 
    By \eqref{eq:c1c1'}, $\sigma_1\le\sigma\le\sigma_1'$. Therefore,
    \begin{equation}\label{eq:inclusions}
    \begin{aligned}
        \{\sigma<\infty\}=&\{\sigma<\infty\}\cap\{\overline M_t>h\overline H_N(U,\widetilde{\mathcal I}_t),\ \forall\ 0\le t<\sigma\}\\\subset&\{\sigma_1<\infty\}\cap\{\overline M_t>h\overline H_N(U,\widetilde{\mathcal I}_t),\ \forall\ 0\le t<\sigma_1\}=:E.
    \end{aligned}
    \end{equation}
    Similarly, 
    \begin{equation}
    \{\sigma<\infty\}\supset E':=\{\sigma_1'<\infty\}\cap\{\overline M_t>h\overline H_N(U,\widetilde{\mathcal I}_t),\ \forall\ 0\le t<\sigma_1'\}.    
    \end{equation}
    
    Write $\mathbb P^x:=\mathbb P[\;\cdot\mid\mphi_{N,0}=h+x\sqrt{\log N}]$. We now estimate the probabilities of $E$ and $E'$ under $\mathbb P^x$, which, combined with the above inclusions, yield desired bounds on $g_{N,r}^h(x)=\mathbb P^x[\sigma<\infty]$. Note that for $0\le s<t\le\sigma$, 
    $$\begin{aligned}
        \langle M\rangle_t-\langle M\rangle_s=\;&{\rm Var}(\overline X_U|\mathcal F_{\wt{\mathcal I}_s})-{\rm Var}(\overline X_U|\mathcal F_{\wt{\mathcal I}_t})\\=\;&\frac1{|U|^2}\sum_{x,y\in U}\left(\wt G_{\partial B_{N+1}\cup\wt{\mathcal I}_s}(x,y)-\wt G_{\partial B_{N+1}\cup\wt{\mathcal I}_t}(x,y)\right)\\=\;&\frac1{|U|}\sum_{y\in U}\sum_{z\in\wt{\mathcal I}_t}\overline H_N(U,z;\wt{\mathcal I_t})\wt G_{\partial B_{N+1}\cup\wt{\mathcal I}_s}(z,y)\\
        \asymp\;&\sum_{z\in\wt{\mathcal I}_t}\overline H_N(U,z;\wt{\mathcal I}_t)H_N(z,\partial B_{N+1};\wt{\mathcal I}_s)
        =\overline H_N(U,\wt{\mathcal I}_t)-\overline H_N(U,\wt{\mathcal I}_s),
    \end{aligned}$$
    where in the ``$\asymp$'' of the fourth line we used the following estimate:
    $$\begin{aligned}
        \frac1{|U|}\sum_{y\in U}\wt G_{\partial B_{N+1}\cup\wt{\mathcal I}_s}(z,y)\asymp\;&\frac 1N\sum_{y\in U}H_N(z,y;U\cup\wt{\mathcal I}_s)\sum_{w\in U}\wt G_{\partial B_{N+1}\cup\wt{\mathcal I}_s}(y,w)\\\asymp\;& H_N(z,U;U\cup\wt{\mathcal I}_s)\asymp H_N(z,\partial B_{N+1};\wt{\mathcal I}_s).
    \end{aligned}$$
        (See \cite[Lemma 20]{DW20} for a similar calculation.)
    Hence, there exist $c_2,c_2'>0$ such that for all $0\le s<t\le\sigma$,
    \begin{equation}\label{eq:har_var}
        c_2(\langle \overline M\rangle_t-\langle \overline M\rangle_s)\leq\overline H_N(U,\widetilde{\mathcal I}_t)-\overline H_N(U,\widetilde{\mathcal I}_s)\leq c'_2(\langle \overline M\rangle_t-\langle \overline M\rangle_s).
    \end{equation}
    Moreover, by Lemma \ref{lem:hitting}, there exist $c_3,c_3'>0$ such that 
    \begin{equation}\label{eq:3.1time0}
        \frac{c_3'}{\log N}\le \overline H_N(U,\widetilde{\mathcal I}_0)\le\frac{c_3}{\log N}.
    \end{equation}
    
    We now prove the main claims. We first assume $h>0$ and prove \eqref{eq:gn+}. For this purpose, we deal with the upper bound of $\mathbb P^x[E]$ first.
    Note that under $\mathbb P^x$, we can write $\overline M_0=\overline H_N(U,\widetilde{\mathcal I}_0)\,\mphi_{N,0}=\overline H_N(U,\widetilde{\mathcal I}_0)\, (h+x\sqrt{\log N})$ by the Markov property of GFF. Combining \eqref{eq:har_var} and \eqref{eq:3.1time0}  with the definition of the event $E$ from \eqref{eq:inclusions}, we get
    \begin{equation}\label{eq:restr_used}
        \overline M_t-\overline M_0>h(\overline H_N(U,\widetilde{\mathcal I}_t)-\overline H_N(U,\widetilde{\mathcal I}_0))-x\sqrt{\log N}\cdot \overline H_N(U,\widetilde{\mathcal I}_0)
        \ge c_2h\langle\overline  M\rangle_t-\frac{c_3x}{\sqrt{\log N}}
    \end{equation}
    for all $0\le t<\sigma_1$. Let $(B_t)_{t\geq 0}$ be a standard Brownian motion. Applying Lemma \ref{lem:dubins}, we have
    $$\begin{aligned}
        \mathbb P^x[E]\le\;&\mathbb P^x\left[\overline M_t-\overline M_0>c_2h\langle\overline M\rangle_t-\frac{c_3x}{\sqrt{\log N}},\ 0\le t<\sigma_1\right]\\
        =\;&\mathbb P\left[B_t>c_2ht-\frac{c_3x}{\sqrt{\log N}},\ 0\le t<c_1|\log r|^{-1}\right]     =\psi(m,b,T),
    \end{aligned}$$
    where $\psi(\cdot)$ was introduced in \eqref{eq:psidef}, $m=c_2h$, $b=\frac{c_3x}{\sqrt{\log N}}$ and $T=c_1|\log r|^{-1}$.
    Moreover, there exists $c>0$ small enough such that if $x\le c\sqrt{\log N}\ |\log r|^{-\frac12}$, then 
    $$b=\frac{c_3x}{\sqrt{\log N}}\le cc_3|\log r|^{-{\frac12}}\le cc_3c_1^{-\frac12}\sqrt{T}\le \sqrt{T},$$
    and by \eqref{eq:asymp+}, we obtain that 
                $$\psi(m,b,T)\lesssim \frac{be^{-c'm^2T}}{\sqrt T}\asymp xe^{-c''K^2}\sqrt{\frac{|\log r|}{\log N}}$$
    for some $c'>0$ and $c''=c'c_1c_2^2$. The upper bound of \eqref{eq:gn+} now follows since $g_{N,r}^h(x)=\mathbb P^x[\sigma<\infty]\le\mathbb P^x[E]$ by \eqref{eq:inclusions}. 

    The derivation for the lower bound of $P^x[E']$ is similar. Now, in order for the event $E'$ to occur, we only need to have $\sigma'_1<\infty$ and $\overline M_t-\overline M_0>h(\overline H_N(U,\widetilde{\mathcal I}_t)-\overline H_N(U,\widetilde{\mathcal I}_0))-x\sqrt{\log N}\cdot \overline H_N(U,\widetilde{\mathcal I}_0)$ holds for all $0\le t<\sigma'_1$, which further contains the event $\overline M_t-\overline M_0>c'_2h\langle\overline  M\rangle_t-\frac{c'_3x}{\sqrt{\log N}}$ by \eqref{eq:har_var} and \eqref{eq:3.1time0}.
    Thus, for some $c>0$ (not necessarily the same as the previous one) and $x\le c\sqrt{\log N}\ |\log r|^{-\frac12}$, we have
    \begin{align*}
        g_{N,r}^h(x)\ge \mathbb P^x[E']\ge\;&\mathbb P^x\left[\overline M_t-\overline M_0>c'_2h\langle\overline M\rangle_t-\frac{c'_3x}{\sqrt{\log N}},\ 0\le t<\sigma'_1\right]\\
        =\;&\mathbb P\left[B_t>c'_2ht-\frac{c'_3x}{\sqrt{\log N}},\ 0\le t<c'_1|\log r|^{-1}\right]\\
        \gtrsim\;&xe^{-c'K^2}\sqrt{\frac{|\log r|}{\log N}}
    \end{align*}
    for some $c'>0$, where we used Lemma \ref{lem:compute} and \eqref{eq:asymp+} in the last line. This completes the proof of \eqref{eq:gn+}.

    We finally turn to \eqref{eq:gn-}, which deals with the case $h\le 0$. In fact, this is similar to the above. For an illustration, to derive the upper bound, we can simply replace \eqref{eq:restr_used} by 
    $\overline M_t-\overline M_0>c_2'h\langle\overline M\rangle_t-\frac{c_3x}{\sqrt{\log N}}$, and conclude with the following bound 
    $$g_{N,r}^h(x)\le \mathbb P^x[E]\le \mathbb P\left[B_t>c'_2ht-\frac{c_3x}{\sqrt{\log N}},\ 0\le t<c_1|\log r|^{-1}\right].$$
    Using Lemma~\ref{lem:compute} again, and taking \eqref{eq:asymp-} into consideration, we can upper bound the right-hand side as required. Since the lower bound of \eqref{eq:gn-} is also similar, we omit its proof and conclude the proof of Proposition \ref{prop:0tok}.
    \end{proof}

Next, we use Proposition \ref{prop:0tok} to finish the proof of Proposition~\ref{thm:metric_graph_bulk}.

\begin{proof}[Proof of Proposition \ref{thm:metric_graph_bulk}]
Noting that when $h<-\sqrt{\log N}$, we have the trivial bound $$\mathbb P[0\connect{\mphi_N\ge h}\partial B_{N/2}]\ge\mathbb P\big[0\xleftrightarrow{\mphi_N\ge -\sqrt{\log N}}\partial B_{N/2}\big],$$
hence we may assume $h\ge-\sqrt{\log N}$ without loss of generality.
Letting 
\begin{equation}\label{eq:xinhdef}
\xi_N^h:=(\mphi_{N,0}-h)/\sqrt{\log N}, 
\end{equation}
we can rewrite the one-arm probability in the following way
\begin{equation}\label{eq:asymp_0tok}
        \mathbb P[0\stackrel{\mphi_N\ge h}{\longleftrightarrow}\partial B_{N/2}]=\mathbb E[g_{N,\frac12}^h(\xi_N^h)\mathbbm1_{\xi_N^h>0}].
    \end{equation}
    Since $\mphi_{N,0}$ is a centered Gaussian with variance of order $\log N$,
        $\xi_N^h$ is Gaussian with mean $-\frac h{\sqrt{\log N}}\le 1$ and variance of constant order. By Proposition \ref{prop:0tok}, there are $c,c'>0$ such that $$g_{N,\frac12}^h(x)\lesssim x\eta_N^{h,c'},\ \forall x\in[0,c\sqrt{\log N}],$$
    where
    \begin{equation}\label{eq:etanhdef}
        \eta_N^{h,\rho}:=\begin{cases}\frac{|h|\vee1}{\sqrt{\log N}}, &h\le 0;\\(\log N)^{-\frac 12}e^{-\rho h^2}, &h>0.\end{cases}
    \end{equation}
    Note that there exists $a>0$ such that $\mathbb P[\xi_N^h>c\sqrt{\log N}]\le N^{-a}$. Therefore,
    $$\begin{aligned}
        \mathbb E[g_{N,\frac12}^h(\xi_N^h)\mathbbm1_{\xi_N^h>0}]\le\;&\mathbb E[g_{N,\frac12}^h(\xi_N^h)\mathbbm1_{0<\xi_N^h\le c\sqrt{\log N}}]+N^{-a}\\
        \lesssim\;&\mathbb E[\xi_N^h\eta_N^{h,c'}\mathbbm1_{0<\xi_N^h\le c\sqrt{\log N}}]+N^{-a}
    \lesssim\eta_N^{h,c'}+N^{-a},
    \end{aligned}$$
        yielding the desired upper bound (by adjusting $c'$ if necessary). 
    
    For the lower bound, we apply Proposition \ref{prop:0tok} again to conclude that there is $c''>0$ such that $g_{N,\frac12}^h(x)\gtrsim x\eta_N^{h,c''}$ for all $x\in[0,1]$. This, combined with the inequality $\mathbb P[0\stackrel{\mphi_N\ge h}{\longleftrightarrow}\partial B_{N/2}]\ge\mathbb E[g_{N,\frac12}^h(\xi_N^h)\mathbbm1_{0<\xi_N^h\le1}]$ from \eqref{eq:asymp_0tok}, gives the desired lower bound. We complete the proof of Proposition~\ref{thm:metric_graph_bulk}.
    \end{proof}

Now, we turn to the boundary one-arm probability. With the help the following proposition, we can obtain Proposition~\ref{thm:metric-boundary} in a similar fashion. We thereby omit the proof of Proposition~\ref{thm:metric-boundary}.
\begin{proposition}\label{prop:0toN}
    There exist $c,\,c'>0$ such that for any $h\ge 0$ and $x>0$, 
    \begin{equation}\label{eq:mkdecay-boundary}
     \mathbb P\Big[0\stackrel{\mphi_N\ge h}{\longleftrightarrow} \partial_i B_N \mid \mphi_{N,0}=h+x\sqrt{\log N}\Big]\le cx(\log N)^{-1}N^{-c' h^2}.
    \end{equation}
                            \end{proposition}

\begin{proof}    Let $(\wt{\mathcal I}_t)_{t\ge0}$ and $\mathbb P^x$ be the same as in the proof of Proposition~\ref{prop:0tok}. Let $V=\partial_i B_N$, and let $M$ be the exploration martingale $M=M_V$. In contrast to the proof of Proposition \ref{prop:0tok}, we adopt the non-normalized observable and its corresponding martingale here to simplify computation. Let $\tau=\inf\{t\ge 0:\widetilde{\mathcal I}_t\cap\partial B_N\neq\emptyset\}$.
    
    By representing the quadratic variance of $M$ in terms of Green's functions (see, e.g., the proof of \cite[Lemma 20]{DW20}), there exist $c_4,c_4'>0$ such that for all $0\le s< t\le\tau$,
    \begin{equation}\label{eq:var_har2}
        c_4(H_N(V,\widetilde{\mathcal I}_t)-H_N(V,\widetilde{\mathcal I}_s))\le\langle M\rangle_t-\langle M\rangle_s\le c_4'(H_N(V,\widetilde{\mathcal I}_t)-H_N(V,\widetilde{\mathcal I}_s)).
    \end{equation}
    where $H_N(\cdot)$ was defined in \eqref{eq:intro_HN}.
    Note that by Lemma \ref{lem:hitting},
    \begin{equation}\label{eq:beg_har}
        H_N(V,\widetilde{\mathcal I}_0)\le \frac{c_5}{\log N}
    \end{equation}
    for some $c_5>0$.
    On the event $\{\tau<\infty\}=\{0\stackrel{\mphi_N\ge h}{\longleftrightarrow}\partial B_N\}$, by first applying \eqref{eq:var_har2} then modifying the proof of \cite[(3.3)]{ding2022crossing} to bound $H_N(V,\widetilde{\mathcal I}_\tau)$ from below, we have
    \begin{equation}\label{eq:end_har}
        \langle M\rangle_\tau\ge c_4(H_N(V,\widetilde{\mathcal I}_\tau)-H_N(V,\widetilde{\mathcal I}_0))\ge c_6\log N
    \end{equation}
    for some $c_6>0$.
    By the definition of the exploration process,
    $$M_t-M_0\ge h(H_N(V,\widetilde{\mathcal I}_t)-H_N(V,\widetilde{\mathcal I}_0))-x\sqrt{\log N}\cdot H_N(V,\widetilde{\mathcal I}_0)$$
     for all $0\le t\leq\tau$. 
    Combining the last inequality with (\ref{eq:var_har2}) and (\ref{eq:beg_har}) gives
    $$M_t-M_0\ge (c_4')^{-1}h\langle M\rangle_t-\frac{c_5x}{\sqrt{\log N}}$$
    for all $0\le t\le\tau$.
    Thus, by (\ref{eq:end_har}),
    \begin{align*}
        \mathbb P^x[\tau<\infty]\le\;&\mathbb P^x\left[M_t-M_0\ge (c_4')^{-1}h\langle M\rangle_t-\frac{c_5x}{\sqrt{\log N}},\forall t\in[0,\tau]\right]\\\le\;&\mathbb P\left[B_t\ge (c_4')^{-1}ht-\frac{c_5x}{\sqrt{\log N}},\forall t\in[0,c_6\log N]\right],
    \end{align*}
    where in the second inequality we used Lemma \ref{lem:dubins}.
    The claim \eqref{eq:mkdecay-boundary} then follows from another application of Lemma \ref{lem:compute} and standard asymptotic analysis.
\end{proof}

\section{The Discrete Case}\label{sec:discrete}
In this section, we focus on the discrete GFF and give the proof of Theorems~\ref{thm:discrete} and \ref{thm:comparison} in Sections~\ref{subsec:discrete-one-arm} and~\ref{subsec:compare}, respectively.

\subsection{Discrete one-arm probability}
\label{subsec:discrete-one-arm}
In this subsection, we prove Theorem \ref{thm:discrete}, which in spirit resembles that of Proposition \ref{thm:metric_graph_bulk} except for an additional issue--in the discrete setup, the counterpart for the key estimate \eqref{eq:terminal_condition} does not exactly hold true. More precisely, in contrast to the metric-gragh level set $\wt E_N^{\ge h}$, when the exploration of $E_N^{\ge h}$ ends before reaching the target $\partial B_{N/2}$, the boundary values of the explored cluster are no longer exactly $h$, but almost surely smaller than $h$. Fortunately, leveraging the correlation of GFF, we are able to show that for all large $h$, the harmonic average of those boundary values is larger than $h-c$ for some constant $c$ with probability sufficiently close to $1$, so that a weaker version of \eqref{eq:terminal_condition} still holds (see Proposition~\ref{prop:cH_bound} for details).

Throughout this subsection, we set $U=\partial B_{\lfloor\frac34N\rfloor}$, and let $(\mathcal I_t)_{t\ge0}$ be the exploration process for $E_N^{\ge h}$ with source $\mathcal I_0=\{0\}$. Take the exploration martingale $\overline M=\overline M_U$. 

We first introduce a technical lemma on the entropic repulsion of 2D GFF, which will be needed both for this subsection (but only the special case when $x\in\Z$, $\delta=1$ and $V=\{x\}$) and the proof of Theorem \ref{thm:comparison}. 
\begin{lemma}\label{lem:pin}
     Let $h\in\mathbb{R}$, $x\in\wt B_{\frac34N}$, $y\in\mZ$ that satisfies $|y-x|=\delta\in(0,1]$, and $V$ be any subset of $\mB_{N}$ containing $x$. Define the event
     \begin{equation}\label{eq:D}
         \mathcal D:=\{\mphi_{N,z}\ge h,\forall z\in V\}\cap\{\mphi_{N,z}<h,\forall z\in B_N\setminus V, |z-x|\ge1 \}.
     \end{equation}
    Then, there exists $c_1>0$ such that for all large $N$,
    \begin{equation}\label{eq:meanpin}
        \mathbb E[\mphi_{N,y}\mid\mathcal D]\ge h-c_1\delta^{\frac12},\quad \forall h\le\sqrt{\log N};
    \end{equation}
    \begin{equation}\label{eq:meanpin+}
        \mathbb E[\mphi_{N,y}\mid\mathcal D]\ge\frac12 h,\quad \forall h>\sqrt{\log N}.
    \end{equation}
\end{lemma}

\begin{proof}
            
    By \cite[Lemma 4.2]{ding2022crossing}, $\mathbb E[\mphi_{N,y}\mid\mathcal D]$ is greater than the one given $\mphi_{N,z}= h$ for all $z\in V$ and the same condition at other sites. Therefore, we will simply replace $\mathcal D$ in \eqref{eq:D} by the following one
    \[
    \mathcal D:=\{\mphi_{N,z} = h,\forall z\in V\}\cap\{\mphi_{N,z}<h,\forall z\in B_N\setminus V, |z-x|\ge1 \},
    \]
    and prove \eqref{eq:meanpin} and \eqref{eq:meanpin+} in this case.

    Let $x\in\wt B_{\frac34N}$, $y\in\mZ$, and $\delta=|y-x|\in(0,1]$ be as in the statement.
    We may assume $y\notin V$, otherwise the estimates trivially hold.   We focus on the case $h\le\sqrt{\log N}$ first and give a detailed proof for \eqref{eq:meanpin}.
    
    First, there is some $c>0$ such that there exists a non-negative continuous function $g_x$ satisfying $g_x(x)=0$, $g_x$ is harmonic on $\mZ\backslash\{x\}$ and \begin{equation}\label{eq:hf}
        \sup_{z\in\mZ}\big|g_x(z)-\log(|z-x|+2)\big|<c.
    \end{equation}
    Let $\kappa>0$ be a parameter to be determined later. Let $f$ be the function defined by 
    (depending on $x$)
\begin{equation}\label{eq:kappadef}
    f(z)=h-\kappa g_x(z).    
    \end{equation}
    Then $f$ is harmonic on $\wt B_{N+1}\backslash\{x\}$, and $f(z)\le h$ for all $z\in V$.
    Moreover, since $h\le\sqrt{\log N}$, by \eqref{eq:hf} and the fact that $x\in\wt B_{\frac34N}$, we can take $N$ large enough such that for all $z\in\partial_i B_N$, 
    \begin{equation}\label{eq:constraint}
        f(z)\le\sqrt{\log N}-\kappa\left(\log (N/4+2)-c\right)<0.
    \end{equation}
    
    Let $\phi$ be another GFF on $\wt B_{N+1}$ with the following conditions $$\phi_z=f(z),\ \forall z\in\partial_i B_{N}\cup V,$$
        and let $\mathcal D'$ be the event
    $$\mathcal D'=\{\phi_z<h,\forall z\in  B_N\setminus V                                                    ,|z-x|\ge1\}.$$
    Then, by \cite[Lemma 4.2]{ding2022crossing},  $\mphi_N$ conditioned on $\mathcal D$ is stochastically larger than $\phi$ conditioned on $\mathcal D'$.
    In particular, for these two fields at the specific point $y$, we have
    \begin{equation}\label{eq:dominate}
        \mathbb E[\mphi_{N,y}|\mathcal D]\ge\mathbb E[\phi_y|\mathcal D'].
    \end{equation}

    From above, it remains to get a corresponding lower bound for $\mathbb E[\phi_y|\mathcal D']$. 
    First, by a union bound and standard Gaussian estimates, we can choose a large constant $\kappa$ (independent of $\delta$ and $N$) in \eqref{eq:kappadef} such that for all large $N$, $\mathbb P[\mathcal D']>\frac12$. Therefore,
    $$(\mathbb E[\phi_y|\mathcal D']-\mathbb E[\phi_y])^2\le 2(\mathbb E[\phi_y|\mathcal D']-\mathbb E[\phi_y])^2\mathbb P[\mathcal D']\le 2{\rm Var}(\phi_y)\lesssim\delta,$$ 
    which yields
    \begin{equation}\label{eq:dif}
        \mathbb E[\phi_y|\mathcal D']\ge \mathbb E[\phi_y]-c\delta^{\frac12}
    \end{equation}
    for some $c>0$.
    Moreover, by the facts that $g_x$ is locally Lipschitz continuous and that $g_x(x)=0$, we have 
    \begin{equation}\label{eq:mean1}
        \mathbb E[\phi_y]=f(y)\ge h-c'\delta
    \end{equation}
    for some $c'=c'(\kappa)>0$.
    Combining \eqref{eq:dif} and \eqref{eq:mean1}, we obtain
    $$\mathbb E[\phi_y|\mathcal D']\ge \mathbb E[\phi_y]-c\delta^{\frac12}\ge h-c_1\delta^{\frac12}$$
    for some $c_1=c_1(\kappa)>0$, which further implies \eqref{eq:meanpin} by \eqref{eq:dominate}. 
    
    To obtain \eqref{eq:meanpin+}, we take $\kappa=ch$ instead, where the constant $c>0$ is chosen sufficiently small such that $f(y)>\frac34h$, and then we can proceed as above. We omit the details.\end{proof}
 
Now we are ready to deduce the substitute for \eqref{eq:terminal_condition}. This requires the use of Lemma~\ref{lem:pin} and Lemma \ref{lem:moment} to control the mean and the fluctuation of the exploration martingale, respectively. 
For integers $1\le k\le N$, we introduce the stopping times $\tau_k$ below
\begin{equation}\label{eq:intro_tauk}
    \tau_k:=\inf\{t\ge 0:\mathcal I_t\cap\partial B_{k}\neq\emptyset\}.
\end{equation}
We further record two special times which will be used frequently below
\[
\sigma:=\tau_{\lfloor N/2\rfloor} \text{ and } \sigma_1:=\tau_{\lfloor\sqrt N\rfloor}.
\] 
For any $x>0$, write as before
    \begin{equation}\label{eq:Pxdef}
    \mathbb P^x:=\mathbb P[\;\cdot\mid\varphi_{N,0}=h+x\sqrt{\log N}],
    \end{equation}
    and denote the corresponding expectation by $\mathbb E^x$.
We have the following estimate.
\begin{proposition}\label{prop:cH_bound}
    For any $h\in\mathbb R$, let $\hat h=\frac h4(1+3\mathbbm1_{h\le\sqrt{\log N}})$. Define the event
    \begin{equation}\label{eq:HNh}
        \mathcal H_N^h:=\{\exists t\in[\sigma_1,\sigma) \mbox{ such that }\overline M_t<(\hat h-3c_1)\,\overline H_N(U,\mathcal I_t)\},
    \end{equation}
    where $c_1$ is from Lemma \ref{lem:pin}. Then, there is some $c>0$ such that for all $x>0$,
    \begin{equation}\label{eq:cH_bound}
        \mathbb P^x[\mathcal H_N^h]\lesssim\begin{cases} N^{-1},&h\le\sqrt{\log N};\\e^{-ch^2},&h>\sqrt{\log N} .\end{cases}
    \end{equation}
\end{proposition}

As mentioned in the beginning of this subsection, this proposition will play the role of \eqref{eq:terminal_condition} in the discrete case, controlling the fluctuation of the exploration martingale around $h$ during the time interval $[\sigma_1,\sigma)$. 
Note that the lack of control over $[0,\sigma_1)$ does not pose any issue, since $\langle \overline M\rangle_{\sigma_1}\lesssim(\log N)^{-1}$ is sufficiently small (see the proof of Proposition~\ref{prop:thm1.3+} below).

\begin{proof}[Proof of Proposition~\ref{prop:cH_bound}]
To obtain \eqref{eq:cH_bound}, we control the exploration martingale $(\overline M_t)_{t\in[\sigma_1,\sigma)}$ along the integer times and times between them separately. More precisely, for an integer time $k$, consider the following two types of events
\begin{equation}\label{eq:intro_Ek}
        E_k:=\left\{\overline M_k<(h-2c_1-\frac58h\mathbbm1_{h>\sqrt{\log N}})\overline H_N(U,\mathcal I_k)\right\},
    \end{equation}
and
\begin{equation}\label{eq:intro_Fk}
        F_k:=\left\{\inf_{t\in[k,k+1]}\overline M_t-\overline M_k<-(c_1+{\frac{1}{8}h\mathbbm1_{h>\sqrt{\log N}}})\cdot \overline H_N(U,\mathcal I_k)\right\}.
\end{equation}
Then, $\mathcal H_N^h$ implies that $E_k\cup F_k$ occurs for some integer $k$ in $[\sigma_1,\sigma)$. 
Apparently, $\overline M_j=\overline M_k$ for all $j>k$ if $\mathcal I_k=\mathcal I_{k-1}$, and $\mathcal I_k\supsetneq \mathcal I_{k-1}$ if $\mathcal I_k\neq \mathcal I_{k-1}$. In other words, the exploration process is ``strictly increasing'' at each step before eventually being stopped. Therefore, there are at most $O(N^2)$ terms to be considered when getting \eqref{eq:cH_bound} by a union bound. In other words, it suffices to show that there exists $c>0$ such that for all $x>0$ and $k\ge1$, 
\begin{equation}\label{eq:cH_bound1}
        \mathbb P^x\left[\{\sigma_1\le k< \sigma\}\cap \left(E_k\cup F_k\right)\right]\lesssim\begin{cases} 
            N^{-3},&h\le\sqrt{\log N};\\
            e^{-ch^2}, &h>\sqrt{\log N}.
            \end{cases}
\end{equation}

    In the remainder of the proof, we fix an integer $k$ and control $E_k$ and $F_k$, separately. 
    We first deal with $E_k$ by conditioning on any possible realization of $\mathcal I_k$.
    Let $I$ be any connected set of vertices such that $0\in I\subset B_{N/2}$. Let $\wt I$ be the metric graph associated with $I$. Assume that $\mathbb P^x[\{\sigma_1\le k< \sigma\}\cap \{\mathcal I_k=\wt I\}]>0$ (note that $\mathcal I_k$ is generally a metric graph as we defined via linear interpolation). It implies that $0$ and $\partial B_{\sqrt N}$ are connected in $I$. 
            Define the random vector $\beta_I$ by $$\beta_I:=( 1_{\{\varphi_{N,x}\ge h\}})_{x\in I}.$$
    The events $\{\mathcal I_k=\wt I\}$ and $\{\sigma_1\le k<\sigma\}$ are both measurable with respect to $\beta_I$. Let $\omega\in\{0,1\}^I$ be some realization of $\beta_I$ such that $\{\beta_I=\omega\}\subset\{\mathcal I_k=\wt I\}$, and define $I^+=\{x\in I:\omega_x=1\}$, $I^-=\{x\in I:\omega_x=0\}$. Note that $0\in I^+$ since $x>0$. 
    
    Now, we estimate the expectation of $\overline M_k$ given $\beta_I$ and $\mathcal F_{I^+}$. 
    Recalling the definition of $D(I)$ in \eqref{eq:intro_D}, we obtain that
        \begin{align} \label{eq:mean} 
        &\;\mathbb E[\overline M_k\mid\beta_I=w,\mathcal F_{I^+}]=\sum_{x\in D(I)}\overline H_N(U,x;I)\mathbb E[\varphi_{N,x}\mid\beta_I=w,\mathcal F_{I^+}]\\ \notag
        \ge&\;\sum_{x\in D(I)\cap I^+}\overline H_N(U,x;I)\cdot h\\ \notag&\; +\sum_{x\in D(I)\cap I^-}\overline H_N(U,x;I)\inf_{z\sim x}\mathbb E[\varphi_{N,x}\mid\varphi_{N,z}\ge h,\varphi_{N,y}<h,\forall y\in B_N,|y-z|\ge1]\\ \notag
        \ge&\;(h-\frac12h\mathbbm1_{h>\sqrt{\log N}}-c_1)\overline H_N(U,I), 
    \end{align}
    where we used 
        Lemma~\ref{lem:pin} with $V=\{z\}$ and $\delta=1$ in the last inequality.
    
    Next, we derive an exponential moment bound for $\overline M_k$ (appropriately recentered and rescaled). Note that $\{\mathcal I_k=\wt I\}$ is non-empty, which indicates that for all $x\in I^-$, there exists some $x'\in I^+$ such that $x\sim x'$. Thus, by Lemma \ref{lem:moment}, for all $r\in\mathbb N$,
    \begin{equation}\label{eq:varM}
        \mu^{(2r)}[\overline M_k|\beta_I=w,\mathcal F_{I^+}]\le(2r-1)!!\cdot 16^r\xi(U,I)^r\overline H_N(U,I)^{2r},
    \end{equation}
    where $\mu^{(2r)}$ stands for the conditional $(2r)$-th moment; see \eqref{eq:intro_muk}. 
    Letting
    $$\zeta_k:=\frac{\overline M_k-\mathbb E[\overline M_k|\beta_I,\mathcal F_{I^+}]}{\xi(U,I)^{\frac12}\overline H_N(U,I)},$$
    it follows from \eqref{eq:varM} that for any $0<b<\frac1{32}$,
    \begin{equation}\label{eq:exp_moment}
        \begin{aligned}
        \mathbb E[e^{b\zeta_k^2}|\beta_I=w,\mathcal F_{I^+}]=&\ \sum_{r=0}^\infty\frac{b^r}{r!\xi(U,I)^r\overline H_N(U,I)^{2r}}\mu^{(2r)}[\overline M_k|\beta_I=w,\mathcal F_{I^+}]\\
        \le&\ \sum_{r=0}^\infty\frac{(16b)^r}{r!}(2r-1)!!\le\sum_{r=0}^\infty(32b)^r<\infty.
        \end{aligned}
    \end{equation}
    Furthermore, we need to give an upper bound for $\xi(U,I)$ (recall its definition in \eqref{eq:xiI}). Write $I'=I\setminus D(I)$. We have for all $y\in U$ and $x\in I'$, 
    $$\begin{aligned}
        H_N(y,x;I')\le&\ H_N(y,x;I')G_{(\wt B_{N+1}\setminus I')\cup\{x\}}(x,x)=G_{(\wt B_{N+1}\setminus I')\cup\{x\}}(y,x)\\=&\ H_N(x,y;(I'\setminus\{x\})\cup\{y\})G_{(\wt B_{N+1}\setminus I')\cup\{x\}}(y,y).
    \end{aligned}$$
    Since $G_{(\wt B_{N+1}\setminus I')\cup\{x\}}(y,y)\le G_{\wt B_{N+1}}(y,y)\lesssim\log N$, and by Lemma~\ref{lem:beurling} (note that $I'$ has diameter at least $\frac12\sqrt N$), for some $a>0$,
            $$H_N(x,y;(I'\setminus\{x\})\cup\{y\})\le P^x\big[\tau_{I'\setminus\{x\}}>\tau_{x+\partial \wt B_{\sqrt N/4}}\big]\lesssim N^{-4a}.$$
    Hence, we have $$\overline H_N(U,x;I')\lesssim \log N\cdot N^{-4a}\lesssim N^{-2a}.$$  Moreover, by Lemma \ref{lem:hitting}, $\overline H_N(U,I)\gtrsim(\log N)^{-1}$. Therefore, 
    \begin{equation}\label{eq:xi_bound1}
        \xi(U,I)=\sup_{x\in I'} \overline H_N(U,x;I')/ \overline H_N(U,I) \lesssim N^{-a}.
    \end{equation}
    
                    When $h\le\sqrt{\log N}$, by \eqref{eq:mean}, \eqref{eq:exp_moment} and \eqref{eq:xi_bound1}, using Chebyshev's inequality, we get
    \begin{equation}\label{eq:step-mart}
        \begin{split}
    &\mathbb P[\overline M_k<(h-2c_1)\overline H_N(U,I)|\beta_I=w,\mathcal F_{I^+}]\\
    \le\;&\mathbb P[|\zeta_k|\ge c_1\xi(U,I)^{-\frac12}]\\\le\;&\exp(-bc_1^2\xi(U,I))     \mathbb E[e^{b\zeta_k^2}|\beta_I=w,\mathcal F_{I^+}]\\
    \lesssim\;& e^{-cN^{a}}\lesssim N^{-3}.
        \end{split}
    \end{equation}
    Similarly, when $h>\sqrt{\log N}$, for some $c>0$,
    \begin{equation}\label{eq:step-mart'}
        \mathbb P[\overline M_k<(\frac38h-2c_1)\overline H_N(U,I)|\beta_I=w,\mathcal F_{I^+}] \le\mathbb P[|\zeta_k|\ge (\frac 18 h+c_1)\xi(U,I)^{-\frac12}] \lesssim e^{-ch^2}.
    \end{equation}
    Combined, we obtain the following desired upper bound for $E_k$,  
    \begin{equation}\label{eq:cH_bound2}
    \begin{aligned}
        &\mathbb P^x[\{\sigma_1\le k<\sigma\}\cap E_k]\\=\;&\mathbb E^x\left[\mathbb P^x[\overline M_k<(h-2c_1-\frac58h\mathbbm1_{h>\sqrt{\log N}}) \overline H_N(U,\mathcal I_k)\mid\mathcal I_k]1_{\{\sigma_1\le k<\sigma\}}\right]\\
        \lesssim\;&\begin{cases}
            N^{-3}, &h\le\sqrt{\log N};\\
            e^{-ch^2}, &h>\sqrt{\log N}.
            \end{cases}
    \end{aligned}
    \end{equation}

    Next, we turn to the event $F_k$.
    First, analogous to (\ref{eq:har_var}), we have the following standard estimate
    \begin{equation}\label{eq:har_var2}
        c_2(\overline H_N(U,\mathcal I_t)-\overline H_N(U,\mathcal I_s))\le \langle \overline M\rangle_t-\langle\overline M \rangle_s\le c_3(\overline H_N(U,\mathcal I_t)-\overline H_N(U,\mathcal I_s))
    \end{equation}
    for some $c_2,c_3>0$ and all $t> s\ge 0$.
    For $\sigma_1\le k<\sigma$, the above inequality and another application of Lemma \ref{lem:beurling} give
    $$\langle \overline M\rangle_{k+1}-\langle \overline M\rangle_k\lesssim \overline H_N(U,\mathcal I_{k+1})-\overline H_N(U,\mathcal I_k)\le \sup_{x\in\mathcal I_{k+1}\setminus\mathcal I_k}H_N(x,\partial B_N;\partial B_N\cup\mathcal I_k)\lesssim N^{-a},$$
    where $H_N(x,\partial B_N;\partial B_N\cup\mathcal I_k):=\sum_{y\in\partial B_N}H_N(x,y;\partial B_N\cup\mathcal I_k)$.
    That is, there exists $c_4>0$ such that
    $$\langle \overline M\rangle_{k+1}-\langle \overline M\rangle_k\le c_4N^{-a}.$$
    Also, note that by Lemma \ref{lem:hitting}, there exists $c_5>0$ such that
    $H_N(U,\mathcal I_k)\ge c_5(\log N)^{-1}$. Recalling $F_k$ from \eqref{eq:intro_Fk} and using Lemma~\ref{lem:dubins}, when $h\le\sqrt{\log N}$,
    \begin{equation}\label{eq:cH_bound3}
        \mathbb P^x[\{\sigma_1\le k<\sigma\}\cap F_k]\le\mathbb P\left[\inf_{t\in[0,c_4N^{-a}]}B_t<-c_1c_5(\log N)^{-1}\right]\lesssim N^{-3}.
    \end{equation}
    Similarly, there is some $c>0$ such that when $h>\sqrt{\log N}$,
    \begin{equation}\label{eq:cH_bound3'}
        \mathbb P^x[\{\sigma_1\le k<\sigma\}\cap F_k]\le\mathbb P\left[\inf_{t\in[0,c_4N^{-a}]}B_t<-(\frac18h+c_1)c_5(\log N)^{-1}\right]\lesssim e^{-ch^2}.
    \end{equation}
    The conclusion (\ref{eq:cH_bound1}) (and hence \eqref{eq:cH_bound}) then follows from (\ref{eq:cH_bound2}), (\ref{eq:cH_bound3}) and (\ref{eq:cH_bound3'}). This finishes the proof.
\end{proof}

To combine the proof of Theorem \ref{thm:discrete} and \eqref{eq:chemical} in the same framework, we derive the following upper bound on the probability of a slightly general version of one-arm event.
\begin{proposition}\label{prop:thm1.3+}
    There is a constant $\rho>0$ such that for any integer $k\in[N(\log N)^{-1},N/2]$ and any $h\le\sqrt{\log N}$, we have
    \begin{equation}
        \mathbb P[0\connect{\varphi_N\ge h}\partial B_k]\lesssim\widetilde \eta_{N,k}^{h,\rho},
    \end{equation}
    where
    $$\wt\eta_{N,k}^{h,\rho}:=\begin{cases}
        \frac{|h|\vee\sqrt{\log N-\log k}}{\sqrt{\log N}},&h\le 0;\\
        \sqrt{\frac{\log N-\log k}{\log N}}\exp\left(-\frac{\rho h^2}{\log N-\log k}\right), &h>0.
    \end{cases}$$
\end{proposition}
\begin{proof}
    As in the proof of Proposition \ref{thm:metric_graph_bulk}, it suffices to deal with the case $h\ge-\sqrt{\log N}$. Furthermore, we may assume that $N$ is larger than any fixed constant. Indeed, arbitrarily fix $c>0$, then for any $N<c$ and $h\in\mathbb R$, there exists $c'>0$ such that we have the following estimate:
    $$\mathbb P\big[0\connect{\varphi_N\ge h}\partial B_k\big]\le\mathbb P[\varphi_{N,0}\ge h]\lesssim\begin{cases}1,&h\le0;\\e^{-c'h^2},&h>0.\end{cases}$$

    As before, the key argument is to estimate the associated exploration martingale.
    Recall $\tau_k$ from \eqref{eq:intro_tauk}. By \eqref{eq:har_var2} and Lemma \ref{lem:hitting}, there are $c_6,c_7>0$ such that on the event $\{0\connect{\varphi_N\ge h}\partial B_{k}\}$,
    \begin{equation}\label{eq:varbound}
    \begin{aligned}
        \langle\overline M\rangle_{\tau_k}\ge\;& c_2(\overline H_N(U,\mathcal I_{\tau_k})-\overline H_N(U,\mathcal I_0))\ge c_6(\log N-\log k)^{-1},\\
        \langle \overline M\rangle_{\sigma_1}\le\;& c_3\,\overline H_N(U,B_{\sqrt N})\le\frac {c_7}{\log N}.
    \end{aligned}
    \end{equation}

    Let $\mathcal H_N^h$ be the event in \eqref{eq:HNh}, and recall $\xi_N^h$ in \eqref{eq:xinhdef}.
    Using  \eqref{eq:har_var2} and Lemma \ref{lem:hitting} again,  there exists $c_8>0$ such that on the event $\{\tau_k<\infty\}\cap (\mathcal H_N^h)^c$, for all $t\in[\sigma_1,\tau_k)$,
    \begin{equation}\label{eq:barMt}
    \begin{aligned}
        \overline M_t-\overline M_0\ge&\ (h-3c_1)\overline H_N(U,\mathcal I_t)-(h+\xi_N^h\sqrt{\log N})\overline H_N(U,\mathcal I_0)\\
        =&\ (h-3c_1)(\overline H_N(U,\mathcal I_t)-\overline H_N(U,\mathcal I_0))-(3c_1+\xi_N^h\sqrt{\log N})\overline H_N(U,\mathcal I_0)\\
        \ge&\ c_8(h-3c_1)\langle \overline M\rangle_t-c_8(3c_1+\xi_N^h\sqrt{\log N})(\log N)^{-1}.
    \end{aligned}
    \end{equation}
    Considering the event
    \begin{equation}\label{eq:calA}
    \begin{aligned}
        \mathcal{A}_k:=\Big\{&\overline M_t-\overline M_0\ge c_8(h-3c_1)\langle\overline  M\rangle_t-\frac {c_8(3c_1+\xi_N^h\sqrt{\log N})}{\log N},\\
    &\forall t \text{ satisfying } \ c_7(\log N)^{-1}<\langle \overline M\rangle_t<c_6(\log N-\log k)^{-1}\Big\},
    \end{aligned}
    \end{equation}
    then \eqref{eq:varbound} and \eqref{eq:barMt} together imply the following inclusion 
    \begin{equation}\label{eq:inclusion}
    \{0\connect{\varphi_N\ge h}\partial B_{k}\} \subset \mathcal H_N^h \cup (\{\tau_k<\infty\}\cap (\mathcal H_N^h)^c) \subset \mathcal H_N^h \cup \mathcal{A}_k.
    \end{equation}
    By Proposition \ref{prop:cH_bound}, $\mathbb P^x[\mathcal H_N^h]\lesssim N^{-1}$, where $\mathbb P^x$ is defined by \eqref{eq:Pxdef}.
        It follows from \eqref{eq:inclusion} that
    \begin{equation}\label{eq:0toN/2_break}
            \mathbb P^x[0\connect{\varphi_N\ge h}\partial B_{k}]
                        \le \mathbb P^x[\mathcal H_N^h]+\mathbb P^x[\mathcal{A}_k] \lesssim N^{-1}+\mathbb P^x[\mathcal{A}_k].
    \end{equation}
    Note that $\mathbb P^x=\mathbb P[\cdot\mid \xi_N^h=x]$ by our notation.
    By Lemmas \ref{lem:compute}, \ref{lem:dubins} and asymptotic estimates \eqref{eq:asymp+}, \eqref{eq:asymp-}, there is $c>0$ such that
    $$\begin{aligned}
        \mathbb P^x[\mathcal{A}_k]=&\,\mathbb P\Big[B_t-B_{\frac{c_7}{\log N}}\ge c_8(h-3c_1)(t-\frac{c_7}{\log N})-\frac{g(x)}{\sqrt{\log N}}-B_{\frac{c_7}{\log N}},\\ &\ \ \ \ \ \forall t\in
        \Big(\frac{c_7}{\log N},\frac{c_6}{\log N-\log k}\Big)\Big]\\
        \lesssim&\,\mathbb E\left[\left(g(x)+\sqrt{\log N}\cdot B_{\frac{c_7}{\log N}}\right)\wt\eta_{N,k}^{h,c}\mathbbm1_{0\le g(x)+\sqrt{\log N}\cdot B_{\frac{c_7}{\log N}}\le\sqrt{\frac{\log N}{\log N-\log k}}}\right]\\&+\mathbb P\left[g(x)+\sqrt{\log N}\cdot B_{\frac{c_7}{\log N}}\ge\sqrt{\frac{\log N}{\log N-\log k}}\right],
    \end{aligned}$$
    where $g(x):=\frac{c_8(-c_7h+3c_1c_7+3c_1+x\sqrt{\log N})}{\sqrt{\log N}}$. 
    Let $\overline\xi$ and $\zeta$ be independent copies of $g(\xi_N^h)$         and $\sqrt{\log N}\cdot B_{\frac{c_7}{\log N}}$, respectively. Then, by \eqref{eq:0toN/2_break} and the above bound for $\mathbb P^x[\mathcal{A}_k]$,
    \begin{align*}
    \mathbb P[0\connect{\varphi_N\ge h}\partial B_k]
            &\lesssim\, N^{-1}+\mathbb P[\mathcal{A}_k]\\
            &\lesssim\, N^{-1}+\wt\eta_{N,k}^{h,c}\,\mathbb E\big[(\overline\xi+\zeta)\mathbbm1_{0\le\overline\xi+\zeta\le\sqrt{\frac{\log N}{\log N-\log k}}}\big]+\mathbb P\left[\overline\xi+\zeta\ge\sqrt{\frac{\log N}{\log N-\log k}}\right]\\
            &\lesssim\, N^{-1}+ \wt\eta_{N,k}^{h,c}+\exp\left(-\frac{c'\log N}{\log N-\log k}\right),
    \end{align*}
        where the last inequality holds since $\overline\xi+\zeta$ is Gaussian with mean and variance of constant order (note that $\mathbb E[\xi_N^h]\in[-1,1]$ since $|h|\le\sqrt{\log N}$). 
        Note that for $h\le 0$, $\eta_N^{h,c}$ is the dominating term on the right-hand side above; while for $0<h\le\sqrt{\log N}$, we may choose $\rho>0$ small enough so that 
    $$\max\left\{N^{-1},\wt\eta_{N,k}^{h,c},\exp\left(-\frac{c'\log N}{\log N-\log k}\right)\right\}\le\wt\eta_{N,k}^{h,\rho}.$$
    The proposition then follows.
\end{proof}

Now, we use Proposition~\ref{prop:thm1.3+} to prove Theorem \ref{thm:discrete}.

\begin{proof}[Proof of Theorem \ref{thm:discrete}]
    To obtain the lower bounds in \eqref{eq:dsub-} and \eqref{eq:dsub+}, just note that
    $$\mathbb P[0\connect{\varphi_N\ge h}\partial B_{N/2}]\ge \mathbb P[0\connect{\mphi_N\ge h}\partial B_{N/2}],$$
    and apply Proposition \ref{thm:metric_graph_bulk}.

    Now we turn to upper bounds. We first deal with case $h\le\sqrt{\log N}$. By Proposition~\ref{prop:thm1.3+}, there exists some $\rho>0$ such that
    $$\mathbb P[0\connect{\varphi_N\ge h}\partial B_{N/2}]\lesssim\wt\eta_{N,N/2}^{h,\rho}\lesssim\eta_N^{h,\rho(\log 2)^{-1}},$$
    where the function $\eta_N^{h,\rho}$ is given by \eqref{eq:etanhdef}. This
    gives the desired bounds.
    When $h>\sqrt{\log N}$, we can proceed as above with $\mathcal{A}_{N/2}$ (defined in \eqref{eq:calA}) replaced by the following event 
    $$\begin{aligned}
        \hat{\mathcal{A}}_{N/2}:=\Big\{&\overline M_t-\overline M_0\ge c_8(\frac12 h-3c_1)\langle\overline  M\rangle_t-\frac {c_8(3c_1+\frac12h+\xi_N^h\sqrt{\log N})}{\log N},\\
    &\forall t \text{ satisfying } \ c_7(\log N)^{-1}<\langle \overline M\rangle_t<c_6(\log N-\log k)^{-1}\Big\},
    \end{aligned}$$
    and apply the following inequality, analogous to \eqref{eq:0toN/2_break},
    $$\mathbb P[0\connect{\varphi_N\ge h}\partial B_{N/2}]\le \mathbb P[\mathcal H_N^h]+\mathbb P[\hat{\mathcal A}_{N/2}].$$
    In fact, by Proposition \ref{prop:cH_bound}, $\mathbb P[\mathcal H_N^h]\lesssim e^{-ch^2}$ for some $c>0$.
    Also, for some $c',c''>0$,
    \begin{equation}
    \begin{split}
        \mathbb P[\hat{\mathcal A}_{N/2}]\;&\le \mathbb P\Big[B_{\frac{c_6}{\log 2}}\ge c_8(\frac12 h-3c_1)\frac{c_6}{\log 2}-\frac {c_8(3c_1+\frac12h+\xi_N^h\sqrt{\log N})}{\log N}\Big]\\
        \;&\le \mathbb P\Big[B_{\frac{c_6}{\log 2}}\ge c'h-c_8\frac{\xi_N^h}{\sqrt{\log N}}\Big]\lesssim e^{-c''h^2}.
    \end{split}    
    \end{equation}
         This yields the desired upper bounds on $\mathbb P[0\connect{\varphi_N\ge h}\partial B_{N/2}]$.
\end{proof}

\subsection*{Sketch of proof of \eqref{eq:chemical}.}
 This estimate on the chemical distance on $E_N^{\ge h}$ can be obtained by adapting the argument from \cite{DW20}, the corresponding result for $\wt E_N^{\ge h}$. All ingredients therein can be easily adapted with the metric-graph GFF replaced by the corresponding arguments for discrete GFF, except for \cite[Lemma 24]{DW20}, whose counterpart for the discrete GFF is stated and proved below.
\begin{lemma}
    For any $h\in\mathbb R$ and $0<\alpha<\beta<\gamma<1$, there exists some $c=c(h,\alpha,\beta,\gamma)>0$ such that for all $j\in\mathbb N$, $k\ge k_0:=(\beta-\alpha)N/\sqrt{\log N}$ and $v\in\partial B_{\alpha N+k}$,
    \begin{equation}\label{eq:additional}
        \mathbb P\left[v\connect{\varphi_N\ge h} B_{\alpha N}\,\Bigg|\,\varphi_{N,v}=h+2^j\sqrt{\log N}\right]\le \frac{c2^j}{\sqrt{\log N}}\sqrt{\log N-\log k}.
    \end{equation}
\end{lemma}
The lemma is proved using the same technique as in the proof of Theorem \ref{thm:discrete}.
\begin{proof}
Write $\mathbb P'$ for the above conditional law. In this proof, we fix $h\in\mathbb R$, and the constants introduced later may implicitly depend on parameters $\alpha,\beta,\gamma$ and $h$. Consider $(\mathcal I_t')_{t\ge0}$, the exploration of $E_N^{\ge h}$ with source $\mathcal I_0'=\{v\}$. Let $k'=k\wedge (\frac{1-\gamma}{2}N)$, $U'=\partial B(v,\frac34(1-\gamma)N)$ and  $M'=\overline M_{U'}$. Define the event 
    $$
    \mathcal H':=\left\{\exists t\in[\sigma_1',\sigma')\mbox{ s.t. }M_t'<(h-3c_1)\overline H_N(U',\mathcal I_t')\right\},
    $$
    where $\sigma'=\inf\{t\ge 0:\mathcal I_t'\cap\partial B(v,k')\neq\emptyset\}$, $\sigma_1'=\inf\{t\ge 0:\mathcal I_t'\cap\partial B(v,\sqrt{N})\neq\emptyset\}$.
    Then, similar to Proposition \ref{prop:cH_bound} (when $h\le\sqrt{\log N}$), we conclude $\mathbb P'[\mathcal H']\lesssim N^{-1}$. Thus, parallel to \eqref{eq:0toN/2_break},
    $$
        \mathbb P'\left[v\connect{\varphi_N\ge h} B_{\alpha N}\right]\le \mathbb P'\left[v\connect{\varphi_N\ge h}\partial B(v,k')\right]
        \le \mathbb P'[\mathcal H']+\mathbb P'[\mathcal A_k'].
   $$
    Here,
    \begin{equation}
       \begin{split}
          \mathcal A_k':=\Big\{&M'_t-M'_0\ge c\langle M'\rangle_t-\frac {c'}{\log N}-\frac{c''2^j}{\sqrt{\log N}},\\
    &\forall t\geq 0\mbox{ such that } c_{-}(\log N)^{-1}<\langle M'\rangle_t<c_{+}(\log N-\log k)^{-1}\Big\}, 
       \end{split} 
    \end{equation}
    where $c,c',c'',c_-,c_+>0$ are appropriately chosen parameter-dependent constants. Applying Lemmas~\ref{lem:compute} and~\ref{lem:dubins}, we obtain the desired estimate.
\end{proof}

\subsection{Comparison with the metric-graph case: proof of Theorem \ref{thm:comparison}}\label{subsec:compare}
To establish the discrepancy of the one-arm probabilities, we show that given the exploration of $\wt E_N^{\ge h}$ from $0$, with the restriction that it stops within the annulus $\wt A_{rN,N/4}$ (which happens with probability of order $(\log N)^{-1/2}$), the conditional probability that the corresponding discrete exploration of $E_N^{\ge h}$ reaching $\partial B_{N/2}$ is bounded away from $0$.

Let $\mathcal C_0$ be the cluster in $\widetilde E_N^{\ge h}$ containing $0$ with the convention that $\mathcal C_0=\emptyset$ if $\mphi_{N,0}<h$. Let $\mathcal C_{\rm ex}^{<h}$ be the union of all clusters in $\widetilde E_N^{< h}$ intersecting $\partial\mB_{N+1}$ with the convention that $\mathcal C_{\rm ex}^{<h}=\emptyset$ if $h\le 0$. 
Define the filling of $\mathcal C_0$ by
\begin{equation}\label{eq:C0*}
    \mathcal C_0^*=\{x\in\wt B_{N+1}:H_N(x,\mathcal C_0)=1\}.
\end{equation}
Let $r\in(0,1/4)$ be a parameter to be determined later. Define the event 
\begin{equation}\label{eq:Lambdadef}
\Lambda=\Lambda_{N,h,r}\coloneq\{\mB_{rN}\subset\mathcal C_0^*\subset\mB_{N/4}\}\cap\{\mathcal C_{\rm ex}^{<h}\cap\mB_{\frac34N}=\emptyset\}.    
\end{equation}
Apparently, we have $\Lambda\subset\{0\disconnect{\mphi_N\ge h}\partial B_{N/2}\}$. We will prove Theorem \ref{thm:comparison} by bounding both $\mathbb P[\Lambda]$ and $\mathbb P[0\connect{\varphi_N\ge h}\partial B_{N/2}|\Lambda]$ from below.
\begin{proposition}\label{prop:mio}
    There exist $r\in(0,1/4)$ and $c=c(r,h)>0$ such that for all $N$,
    \begin{equation}\label{eq:lambda}
        \mathbb P[\Lambda]\ge\frac c{\sqrt{\log N}}.
    \end{equation}
\end{proposition}
\begin{proof}
    Recall the definition of the event $\mathcal C_{N,k,l}^h$ from above Proposition \ref{prop:circuit'}. It suffices to give a lower bound on the probability of the event
    $$\Lambda':=\mathcal C_{N,\frac34N,N}^h\cap\mathcal C_{N,rN,2rN}^h\cap\left\{0\connect{\mphi_N\ge h}\partial \mB_{2rN},0\disconnect{\mphi_N\ge h}\partial \mB_{N/4}\right\},$$
    since apparently, $\Lambda'\subset \Lambda$.
    By Proposition \ref{prop:0tok} and an integration similar to the proof of Proposition~\ref{thm:metric_graph_bulk}, there exist  $c(h),c'(h)>0$ such that
    $$c\sqrt{\frac{|\log r|}{\log N}}<\mathbb P[0\connect{\mphi_N\ge h}\partial B_{2rN}]<c'\sqrt{\frac{|\log r|}{\log N}}.$$
    By Proposition~\ref{prop:circuit'}, there exists $c''=c''(r,h)>0$ such that
    $$\min\left\{\mathbb P[\mathcal C_{N,\frac34N,N}^h],\mathbb P[\mathcal C_{N,rN,2rN}^h]\right\}>c''.$$
    Thus, by FKG inequality, one may choose $r\in(0,\frac14)$ sufficiently small such that
    \begin{align*}
        \mathbb P[\Lambda']\ge\;& \mathbb P[\mathcal C_{N,\frac34N,N}^h]\cdot\mathbb P[\mathcal C_{N,rN,2rN}^h]\cdot\mathbb P[0\connect{\mphi_N\ge h}\partial B_{2rN}]-\mathbb P[0\connect{\mphi_N\ge h}\partial B_{N/2}]\\
        \ge\;&\left((c'')^2c\sqrt{|\log r|}-c'\sqrt{\log4}\right)\cdot(\log N)^{-1/2}\gtrsim\frac{1}{\sqrt{\log N}}.
    \end{align*}
    The conclusion then follows since $\Lambda'\subset \Lambda$.
\end{proof}

From now on till the end of this section, we choose some $r$ (in the definition of $\Lambda=\Lambda_{N,h,r}$ in \eqref{eq:Lambdadef}) such that \eqref{eq:lambda} holds. 
To lighten notation, we use $P^*$ to denote the conditional probability given all the information of GFF at $\mathcal C_0\cup\mathcal C_{\rm ex}^{<h}$ and the occurrence of the event $\Lambda$, that is, 
\begin{equation}\label{eq:P*}
    \mathbb P^*:=\mathbb P\Big[\;\cdot\,\Big |\, \mathcal F_{\mathcal C_0\cup\mathcal C_{\rm ex}^{<h}}, \Lambda\Big],
\end{equation}
and furthermore use $\mathbb E^*$ and ${\rm Var}^*$ to denote the corresponding mean and variance, respectively. 
Recall $\mathcal C_0^*$ in \eqref{eq:C0*} and define 
\begin{equation}\label{eq:Gamma}
    \Gamma:=\{x\in\Z:1\le{\rm dist}(x,\cC_0^*)<2\}.
\end{equation}
For any $A\subset B\subset B_N\setminus(\mathcal C_0\cup\mathcal C_{\rm ex}^{<h})$, write 
\[
\overline H_N^*(\partial B_{\frac58 N},A;B)=\overline H_N(\partial B_{\frac58N},A;B\cup\mathcal C_0\cup\mathcal C_{\rm ex}^{<h}),
\]
and abbreviate $\overline H^*_N(\partial B_{\frac58N},A)$ when $A=B$.
\begin{proposition}\label{prop:gamma_to_N/2}
    For all $h\in \mathbb R$, there exists $c=c(h)>0$ such that 
    \begin{equation}\label{eq:P*Gamma}
        \mathbb P^*[\Gamma\connect{\varphi_N\ge h}\partial B_{N/2}]\ge c.
    \end{equation}
\end{proposition}
\begin{proof}
    The proof is based on the same ideas presented in \cite[Section 4]{ding2022crossing}. 
    We assume that $\Lambda$ holds below, which is measurable with respect to $\mathcal C_0\cup\mathcal C_{\rm ex}^{<h}$.  
    
    We first address the case where $h\ge 0$. Conditionally on $\mathcal F_{\mathcal C_0\cup\mathcal C_{\rm ex}^{<h}}$, let $(\mathcal I_t)_{t\ge0}$ be the exploration of $E_N^{\ge h}$ with source $\mathcal I_0=\Gamma$, and consider the martingale $\overline M_t:=\mathbb E^*\big[\overline X_{\partial B_{\frac58N}} \big| \mathcal F_{\mathcal I_t}\big]$. Let $\tilde E$ denote the event $\Big\{\Gamma\disconnect{\varphi_N\ge h}\partial B_{N/2}\Big\}$, and define the stopping time $\sigma':=\inf\{t\ge0:\mathcal I_t\cap\partial B_{N/2}\neq\emptyset\}$. We claim that there exists $\Delta>0$, such that
    \begin{equation}\label{eq:P*F}
        \mathbb P^*[F\mid\tilde E]=1-o_N(1),
    \end{equation}
    where $$F:=\{\exists 0<t<\sigma'\ \mbox{ s.t. }\overline M_t<h-\Delta\cdot \overline H_N^*(\partial B_{\frac58N},\mathcal I_t)\}.$$
    Note that $\Lambda$ entails $\Gamma\subset B_{N/2}$ and $\mathcal C_{\rm ex}^{<h}\cap\wt B_{\frac34N}=\emptyset$.
    Therefore, we have $\overline H_N^*(\partial B_{\frac58N},\Gamma)\asymp 1$ by Lemma \ref{lem:hitting},
    which in turn gives 
    $$\begin{aligned}
        {\rm Var^*}(\overline M_0)=&\ \sum_{x,y\in\Gamma}\overline H_N^*(\partial B_{\frac58N},x;\Gamma)\overline H_N^*(\partial B_{\frac58N},y;\Gamma)G_{\wt B_{N+1}\setminus(\mathcal C_0\cup\mathcal C_{\rm ex}^{<h})}(x,y)\\
        \lesssim&\ \overline H_N^*(\partial B_{\frac58N},\Gamma)^2\asymp \overline H_N^*(\partial B_{\frac58N},\Gamma).
    \end{aligned}$$
    Moreover, similar to \eqref{eq:har_var}, we have for $0<t<\sigma'$,
    $$\overline H_N^*(\partial B_{\frac58N},\mathcal I_t)-\overline H_N^*(\partial B_{\frac58N},\Gamma)\gtrsim\langle\overline M\rangle_t.$$
    Consequently, there is some constant $c>0$ such that for all $0<t<\sigma'$,
    $$\begin{aligned}
        \overline H_N^*(\partial B_{\frac58N},\mathcal I_t)=&\ \big(\overline H_N^*(\partial B_{\frac58N},\mathcal I_t)-\overline H_N^*(\partial B_{\frac58N},\Gamma)\big)+\frac12\overline H_N^*(\partial B_{\frac58N},\Gamma)+\frac12\overline H_N^*(\partial B_{\frac58N},\Gamma)\\
        \ge&\ c(\langle\overline M\rangle_t+{\rm Var}^*(\overline M_0)+1),
    \end{aligned}$$
    further yielding
    $F\subset\{\exists 0<t<\sigma'\ \mbox{ s.t. }\overline M_t-h<-c\Delta(\langle\overline M\rangle_t+{\rm Var}^*(\overline M_0)+1) \}$.
    {Applying Lemma \ref{lem:dubins} again, we may view $\overline M_t-\overline M_0$ as a time-changed standard Brownian motion. Further noting that $\overline M_0-h$ is a centered Gaussian variable with variance ${\rm Var}^*(\overline M_0)$, we may view the process $\overline M_t-h=(\overline M_t-\overline M_0)+(\overline M_0-h)$ as the part of a standard Brownian motion after time ${\rm Var}^*(\overline M_0)$. Hence, }we obtain
    $$\mathbb P^*[F]\le\mathbb P[\exists t\ge0\ \mbox{ s.t. }B_t\le-c\Delta(t+1
    )]=:1-p_0,$$
    where $p_0>0$ by Lemma~\ref{lem:compute} (sending $T\to\infty$). 
    This combined with \eqref{eq:P*F} gives that for large $N$,
    $$\mathbb P^*[\tilde E]\le\frac{\mathbb P^*[F]}{\mathbb P^*[F|\tilde E]}<1-\frac12p_0.$$
    This implies \eqref{eq:P*Gamma} immediately.

    We next turn to the proof of \eqref{eq:P*F}. Note that it suffices to bound from below the probability of $F$ given $\{\mathcal I_\infty=\wt I\}$, where $I$ is any connected set of vertices satisfying $\Gamma\subset I\subset B_{N/2}\setminus\mathcal C_0$ and $\wt I$ is the metric graph associated with $I$. Moreover, $\{\mathcal I_\infty=\wt I\}$ is measurable with respect to $\beta_I:=(\mathbbm 1_{\{\varphi_{N,x}\ge h\}})_{x\in I}$, then it suffices to show on $\{\mathcal I_\infty=I\}$,
    \begin{equation}\label{eq:cond_beta}
        \mathbb P^*[\overline M_\infty<(h-\Delta)\overline H_N^*(\partial B_{\frac58N},I)\mid\beta_I]=1-o_N(1),\;\; {\rm a.s.}
    \end{equation}
    The strategy for the proof of \eqref{eq:cond_beta} given below is in fact very similar to \eqref{eq:cH_bound1} (but note that the desired inequality is in a reverse direction).
    
    Consider any $b_I\in\{0,1\}^I$ such that $\{
    \beta_I=b_I\}\subset\{\mathcal I_\infty=I\}$. Let $I^+=\{x\in I:b_{I,x}=1\}$, $I^-=I\setminus I^+$, and recall the definition of $D(I)$ from \eqref{eq:intro_D}. Note that we have $D(I)\subset I^-$. Moreover, on the event $\{\beta_I=b_I\}$, under $\mathbb P^*$,
    \begin{equation}\label{eq:M_infty}
        \overline M_\infty=\sum_{x\in D(I)}\overline H_N^*(\partial B_{\frac58N},x;I)\varphi_{N,x}.
    \end{equation}
    By \cite[Lemma 4.2]{ding2022crossing} (comparison of GFFs with different conditions), we get that for all $x\in D(I)$,
    \begin{equation}\label{eq:Ephi}
        \begin{split}
            &\mathbb E^*[\varphi_{N,x}\mid\beta_I=b_I,\mathcal F_{I^+}]\le \mathbb E^*[\varphi_{N,x}\mid\varphi_{N,x}< h,\, \beta_{I^+}=b_{I^+},\, \mathcal F_{I^+}]\\
            \le\;& \mathbb E^*[\varphi_{N,x}\wedge h\mid \beta_{I^+}=b_{I^+},\, \mathcal F_{I^+}]\le h-ce^{-c'm_x^2},
        \end{split}
    \end{equation}
    where constants $c,c'>0$ and 
    \[
    m_x:=\mathbb E^*[\varphi_{N,x}\mid\beta_{I^+}=b_{I^+},\, \mathcal F_{I^+}]-h>0.
    \]
    We claim that $Y':=\sum_{x\in D(I)}\overline H_N^*(\partial B_{\frac58N},x;I)m_x$ satisfies
    \begin{equation}\label{eq:Yprime_upper}
    \mathbb P^*[Y'\le c''\overline H_N^*(\partial B_{\frac58N},I)\mid\beta_I=b_I]=1-o_N(1)
    \end{equation}
    for some constant $c''>0$.
    
    We now prove \eqref{eq:cond_beta} assuming \eqref{eq:Yprime_upper} for the time being.
    By Markov's inequality, $Y'\le c''\overline H_N^*(\partial B_{\frac58N},I)$ implies
    \begin{equation}\label{eq:mx}
        \sum_{x\in D(I)}\overline H_N^*(\partial B_{\frac58N},x;I)\mathbbm1_{m_x\le 2c''}\ge\frac12\overline H_N^*(\partial B_{\frac58N},I).
    \end{equation}
    Combining \eqref{eq:M_infty}, \eqref{eq:Ephi} and \eqref{eq:mx}, we obtain that with probability $1-o_N(1)$,
    $$\begin{aligned}
        \mathbb E^*[\overline M_\infty\mid\beta_I=b_I,\mathcal F_{I^+}]=&\ \sum_{x\in D(I)}\overline H_N^*(\partial B_{\frac58N},x;I)\,\mathbb E^*[\varphi_{N,x}\mid\beta_I=b_I,\mathcal F_{I^+}]\\\le&\ h\overline H_N^*(\partial B_{\frac58N},I)-\sum_{x\in D(I)}ce^{-c'm_x^2}\overline H_N^*(\partial B_{\frac58N},x;I)\\
        \le&\ h\overline H_N^*(\partial B_{\frac58N},I)-ce^{-c'(2c'')^2}\sum_{x\in D(I)}\overline H_N^*(\partial B_{\frac58N},x;I)\mathbbm1_{m_x\le 2c'}\\
        \le&\ (h-2\Delta)\overline H_N^*(\partial B_{\frac58N},I),
    \end{aligned}$$
    where we set $\Delta=\frac14ce^{-c'(2c'')^2}$. 
    Moreover, it is not hard to check that both \eqref{eq:varM} and \eqref{eq:xi_bound1} hold under the measure $\mathbb P^*$ for $I$ in the current context, which further imply that
    \begin{equation}\label{eq:VarM}
        {\rm Var}^*(\overline M_\infty|\beta_I=b_I,\mathcal F_{I^+})=o_N(1)\overline H_N^*(\partial B_{\frac58N},I)^2.
    \end{equation}
    It follows from Chebyshev's inequality that
    \begin{align*}
        \mathbb P^*[\overline M_\infty > (h-\Delta)\overline H_N^*(\partial B_{\frac58N},I)\mid\beta_I]&\le \mathbb P^*[\overline M_\infty - \mathbb E^*[\overline M_\infty\mid\beta_I] > \Delta \overline H_N^*(\partial B_{\frac58N},I)\mid\beta_I]\\
        &\le \frac{o_N(1)\overline H_N^*(\partial B_{\frac58N},I)^2}{(\Delta \overline H_N^*(\partial B_{\frac58N},I))^2}=o_N(1),
    \end{align*}
    finishing the proof of \eqref{eq:cond_beta} and hence of \eqref{eq:P*Gamma}.
    
    It remains to show \eqref{eq:Yprime_upper}. Let $D'=\{x:\overline H_N^*(\partial B_{\frac58 N},x;I^+)>0\}$ and note that $$Y'=\sum_{x\in D'}\overline H_N^*(\partial B_{\frac58N},x;I^+)(\varphi_{N,x}-h).$$
    The desired bound on $Y'$ then follows by adapting the proof of \eqref{eq:cond_beta}. First, observe that every vertex in $D'$ has a nearest neighbor in $I^-$, so an adaption of Lemma \ref{lem:pin} yields that for some $c>0$,
    $$\mathbb E^*[Y'|\beta_I=b_I,\mathcal F_{I^-}]\le c\overline H_N^*(\partial B_{\frac58N},I^+)\le c\overline H_N^*(\partial B_{\frac58N},I).$$
    Second, similar to \eqref{eq:VarM}, we have
    $${\rm Var}^*(Y'|\beta_I=b_I,\mathcal F_{I^-})\le o_N(1)\overline H_N^*(\partial B_{\frac58N},I)^2.$$
    We then conclude by applying Chebyshev's inequality again.

    For the case $h<0$, we consider a new field $\mphi_{N}^-:=\mphi_N-f_h$, where $f_h(x)=|h|\overline H_N^*(x,\partial B_{N+1})$ for all $x\in\wt B_{N+1}\setminus(\mathcal C_0\cup\mathcal C_{\rm ex}^{<h})$, and denote its restriction to the lattice by $\varphi_N^-$. Then, $\mphi_N^-$ is a GFF on $\wt B_{N+1}\setminus(\mathcal C_0\cup\mathcal C_{\rm ex}^{<h})$ with boundary conditions $h$, and we have $\varphi_N\ge \varphi_N^-$. Therefore, it suffices to show that $\mathbb P^*[\Gamma\connect{\varphi_N^-\ge h}\partial B_{N/2}]\ge c$ for some $c>0$. This follows from the same argument as above. We conclude the proof of Proposition \ref{prop:gamma_to_N/2}.
\end{proof}
We finally turn to Theorem \ref{thm:comparison}. Proposition \ref{prop:gamma_to_N/2} implies that with uniformly positive $\mathbb P^*$-probability (recall that $\mathbb P^*$ is defined in \eqref{eq:P*}), there exists a ``pivotal vertex'' for the event $\{0\connect{\varphi_N\ge h}\partial B_{N/2}\}$. Thus, it remains to prove that the cost of ``opening'' that pivotal vertex is a constant (c.f.\ \eqref{eq:givenbeta*} below). 
\begin{proof}[Proof of Theorem \ref{thm:comparison}]

    Recall $\mathcal C_0^*$ and $\Gamma$ in \eqref{eq:C0*} and \eqref{eq:Gamma} respectively, and further define 
\[
\Gamma':=\{x\in\Z:0<{\rm dist}(x,\mathcal C_0^*)<1\}.
\] 
    Note that on $\{\Gamma\connect{\varphi_N\ge h}\partial B_{N/2}\}$, we may choose some random vertex $v\in\Gamma'$ such that $0$ is connected to $\partial B_{N/2}$ in $E_N^{\ge h}$ whenever $\varphi_{N,v}\ge h$ (we pick one deterministically if there are multiple candidates).
    Let
    $$\beta^*:=(\mathbbm 1_{\varphi_{N,x}\ge h})_{x\in B_N\setminus(\mathcal C_0\cup \mathcal C_{\rm ex}^{<h}\cup\Gamma')}.$$
    Then $\{\Gamma\connect{\varphi_N\ge h}\partial B_{N/2}\}$ and the site $v$ are measurable with respect to $\beta^*$. We now claim that it suffices to show that there exists $c=c(h)>0$ such that 
    \begin{equation}\label{eq:givenbeta*}
        \mathbb P^*[\varphi_{N,v}\ge h|\beta^*]\ge c.
    \end{equation}
    Assuming \eqref{eq:givenbeta*}, one has,
    \begin{equation*}
    \begin{split}
        \mathbb P^*[0\connect{\varphi_N\ge h}\partial B_{N/2}]&\;\ge\mathbb P^*[\Gamma\connect{\varphi_N\ge h}\partial B_{N/2},\ \varphi_{N,v}\ge h]\\&\;=\mathbb E^*\left[\mathbbm1_{\{\Gamma\connect{\varphi_N\ge h}\partial B_{N/2}\}}\mathbb P^*[\varphi_{N,v}\ge h\mid\beta^*]\right]\ge c,
    \end{split}    
    \end{equation*}
     where the last step follows from \eqref{eq:givenbeta*} and Proposition \ref{prop:gamma_to_N/2}.
    Hence, $\mathbb P\big[0\connect{\varphi_N\ge h}\partial B_{N/2}\big|\Lambda\big]\ge c$, which, combined with Proposition \ref{prop:mio}, yields
    $$\mathbb P\big[0\connect{\varphi_N\ge h}\partial B_{N/2}\big]-\mathbb P\big[0\connect{\mphi_N\ge h}\partial B_{N/2}\big]\ge\mathbb P[\Lambda]\ \mathbb P\big[0\connect{\varphi_N\ge h}\partial B_{N/2}\big|\Lambda\big]\ge \frac{c}{\sqrt{\log N}}$$
    for some $c=c(h)>0$, proving the theorem.
    
    We now prove \eqref{eq:givenbeta*}. Let $\delta={\rm dist}(v,\mathcal C_0^*)$. By the Markov property of the GFF, we have the orthogonal decomposition
    $\varphi_{N,v}=Y+Z$, where $Y$ is a centered Gaussian variable independent of $\mathcal F_{B_N\setminus(\mathcal C_0\cup\mathcal C_{\rm ex}^{<h}\cup\Gamma')}$, such that ${\rm Var}^*(Y)\ge c\delta$ for some universal constant $c>0$, and $Z=\mathbb E^*[\varphi_{N,v}\mid\mathcal F_{B_N\setminus(\mathcal C_0\cup\mathcal C_{\rm ex}^{<h}\cup\Gamma')}]$. By Lemma \ref{lem:pin}, $\mathbb E^*[Z\mid\beta^*]\ge h-c_1\delta^{\frac12}$.
    On the other hand, it is not hard to see
    ${\rm Var}^*(Z\mid\beta^*)\le{\rm Var}^*(Z)\le c'\delta^2$ for some universal $c'<\infty$. Combined, 
    we get that both $Y\ge 2c_1\delta^{\frac12}$ and $Z\ge h-2c_1\delta^{\frac12}$ occur with positive probability under $\mathbb P^*[\cdot\mid\beta^*]$. The combination of these two events gives $\varphi_{N,v}\ge h$, implying \eqref{eq:givenbeta*} and finishing the proof.  
\end{proof}

\section*{Acknowledgments}
We would like to thank Jian Ding and Zhenhao Cai for inspiring discussions. YB and XL are supported by National Key R\&D Program of China (No.\  2021YFA1002700). YG is supported by the startup fund from Westlake University.

	\bibliographystyle{plain}
	\bibliography{references}

\end{document}